\documentclass[reqno, english, 12pt]{amsart}
\pdfoutput = 1 

\usepackage[usenames, dvipsnames]{xcolor} 
\usepackage{amsfonts,amsmath,amssymb,amsthm,bbm,mathtools,comment}
\usepackage[shortlabels]{enumitem}

\usepackage[hyphens]{url} 
\usepackage[linktoc = all, hidelinks, colorlinks, unicode=true, pagebackref=true]{hyperref} 
\usepackage[capitalise, compress, nameinlink, noabbrev]{cleveref} 
\hypersetup{linkcolor={blue!70!black}, citecolor={green!70!black}, urlcolor={blue!70!black}}

\usepackage[mathscr]{euscript}
\usepackage{adjustbox,tikz,calc,graphics,babel,standalone}
\usetikzlibrary{shapes.misc,calc,intersections,patterns,decorations.pathreplacing}
\usetikzlibrary{arrows,shapes,positioning}
\usetikzlibrary{decorations.markings,quotes,arrows.meta}
\usepackage{tikz}
\usepackage[final]{microtype}
\usepackage[numbers]{natbib}
\usepackage{cmtiup}
\usepackage{graphicx}
\usepackage{caption}
\usepackage{subcaption}
\usepackage{verbatim}
\usepackage{array}
\usepackage[frame,cmtip,arrow,matrix,line,graph,curve]{xy}
\usepackage{graphpap,pstricks}
\usepackage{pifont}
\usepackage{cmtiup}
\usepackage{todonotes}
\usetikzlibrary{topaths,calc} 
\tikzstyle{vertex} = [fill,shape=circle,node distance=80pt]
\tikzstyle{edge} = [opacity=0.4,fill opacity=0.0,line cap=round, line join=round, line width=40pt]
\tikzstyle{elabel} =  [fill,shape=circle,node distance=30pt]
\tikzstyle{circ} = [draw, fill=SpringGreen, circle, inner sep=0.12cm]
\tikzstyle{square} = [draw, fill=RedOrange, rectangle, inner sep=0.15cm]
\tikzstyle{triangle} = [draw, fill=Turquoise, regular polygon, regular polygon sides=3, inner sep=0.08cm]

\usepackage[margin = 2.3cm,foot=1cm]{geometry}

\allowdisplaybreaks

\theoremstyle{plain}
\newtheorem{theorem}{Theorem}[section]		
\newtheorem{lemma}[theorem]{Lemma}
\newtheorem{claim}[theorem]{Claim}
\newtheorem{proposition}[theorem]{Proposition}
\newtheorem{corollary}[theorem]{Corollary}
\newtheorem{conjecture}[theorem]{Conjecture}

\newtheorem{example}[theorem]{Example}
\newtheorem{question}[theorem]{Question}
\newtheorem{definition}[theorem]{Definition}

\theoremstyle{remark}
\newtheorem*{remark}{Remark}

\newcommand{\cC}{\mathcal{C}}
\newcommand{\cF}{\mathcal{F}}
\newcommand{\cH}{\mathcal{H}}
\newcommand{\cO}{\mathcal{O}}
\newcommand{\cA}{\mathcal{A}}
\newcommand{\cT}{\mathcal{T}}
\newcommand{\cE}{\mathcal{E}}

\newcommand{\bZ}{\mathbb{Z}}

\DeclareMathOperator{\poly}{poly}
\DeclareMathOperator{\RS}{RS}
\DeclareMathOperator{\Eq}{\cE}
\DeclareMathOperator{\Eqs}{\cE}
\DeclareMathOperator{\wrap}{wrap}

\DeclarePairedDelimiter{\abs}{\lvert}{\rvert}
\DeclarePairedDelimiter{\set}{\{}{\}}

\renewcommand{\emptyset}{\varnothing}
\renewcommand{\le}{\leqslant}
\renewcommand{\leq}{\leqslant}
\renewcommand{\ge}{\geqslant}
\renewcommand{\geq}{\geqslant}

\newcommand{\eps}{\ensuremath{\varepsilon}}

\newcommand{\defn}[1]{\textcolor{Maroon}{\emph{#1}}}

\newcommand{\e}{\ensuremath{\varepsilon}}

\newcommand{\cP}{\mathcal{P}}

\let\originalleft\left
\let\originalright\right
\renewcommand{\left}{\mathopen{}\mathclose\bgroup\originalleft}
\renewcommand{\right}{\aftergroup\egroup\originalright}

\makeatletter
\def\imod#1{\allowbreak\mkern10mu({\operator@font mod}\,\,#1)}
\makeatother

\newcommand{\urlprefix}{}
\title{Abundance: Asymmetric Graph Removal Lemmas and Integer Solutions to Linear Equations}
\author{Ant\'onio Gir\~ao}
\author{Eoin Hurley}
\author{Freddie Illingworth}
\author{Lukas Michel}

\thanks{
AG \& LM: Mathematical Institute, University of Oxford, Oxford OX2 6GG, UK. E-mail: \textsf{\{girao,michel\}\allowbreak@maths.ox.ac.uk}. Research of AG supported by EPSRC grant EP/V007327/1.}
\thanks{
EH: Korteweg de Vries Instituut, Universiteit van Amsterdam,  Science Park 904, 1098 XH Amsterdam , The Netherlands. E-mail: \textsf{eoin.hurley@umail.ucc.ie}}
\thanks{FI: Department of Mathematics, University College London, 25 Gordon Street, WC1H 0AY, UK. E-mail: \textsf{f.illingworth@ucl.ac.uk}}

\begin{document}

\maketitle

\begin{abstract}
    We prove that a large family of pairs of graphs satisfy a polynomial dependence in asymmetric graph removal lemmas. In particular, we give an unexpected answer to a question of Gishboliner, Shapira, and Wigderson~\cite{GSW23} by showing that for every $t \geq 4$, there are $K_t$-abundant graphs of chromatic number $t$.
    Using similar methods, we also extend work of Ruzsa~\cite{Ruzsa} by proving that a set $\cA \subset \set{1,\dots,N}$ which avoids solutions with distinct integers to an equation of genus at least two has size $\cO(\sqrt{N})$.
    The best previous bound was $N^{1 - o(1)}$ and the exponent of $1/2$ is best possible in such a result.
    Finally, we investigate the relationship between polynomial dependencies in asymmetric removal lemmas and the problem of avoiding integer solutions to equations. 
    The results suggest a potentially deep correspondence. Many open questions remain.
\end{abstract}

\section{Introduction}

Graph theory and additive combinatorics are distinct areas of discrete mathematics with a deep and interesting interplay. 
One particularly elegant example is the interplay between the \emph{triangle removal lemma} and sets of integers avoiding $3$-term arithmetic progressions.
The former is an influential result in extremal graph theory proved in the seminal paper of Ruzsa and Szemer\'{e}di~\cite{RS78}. It states: 
\begin{quote}
    If an $n$-vertex graph has $o(n^3)$ triangles, then all triangles can be deleted by removing $o(n^2)$ edges.
\end{quote}
Contrapositively, if an $n$-vertex graph requires the deletion of at least $\e n^2$ edges to be made triangle-free, then it contains at least $\delta n^3$ triangles where $\delta = \delta(\eps) > 0$ depends only on $\eps$. 
In many applications it would be useful if the function $\delta(\eps)$ is not too small and, in particular, if it is at least some polynomial.
Unfortunately, this is not the case.
 Ruzsa and Szemer\'{e}di made brilliant use of  Behrend's construction~\cite{Behrend} of large sets $\cA\subset [N]\coloneqq \set{1,\dots,N}$ with no non-trivial $3$-term arithmetic progressions to prove a sub-polynomial upper bound on $\delta(\eps)$. 
If one is interested in removal lemmas for graphs other than triangles, then one must consider large sets of integers avoiding additive structures other than $3$-term arithmetic progressions. 
This provides a link between polynomial dependence in removal lemmas and the question: 
\begin{quote}
    How large can a set of integers $\cA\subset [N]$ be if it has no non-trivial solutions to some linear diophantine equation?
\end{quote}
Such problems are central to additive combinatorics and were systematically studied by Ruzsa~\cite{Ruzsa} who investigated the properties of equations $E$ that determine the size of the largest subset of $[N]$ containing no non-trivial solutions to $E$.

{In this paper we make progress in three directions: 
\begin{itemize}
    \item We develop a number of tools to construct graphs with polynomial dependencies of $\delta$ on $\eps$ in \emph{asymmetric} graph removal lemmas. 
    This yields a surprising answer to a question posed by Gishboliner, Shapira and Wigderson~\cite{GSW23}, uncovering richer behaviour than anticipated.
    \item We extend the work of Ruzsa~\cite{Ruzsa} by proving that a set $\cA \subset [N]$ which avoids solutions with distinct integers to an equation of genus at least two has size $\cO(\sqrt{N})$.
    \item We probe the relationship between these two problems, proving both positive and negative results. 
\end{itemize}}  
We now provide a miniature introduction for each direction, before proving the relevant results for the three bullet points in \cref{sec:Abundance,sec:additive,,sec:eqtographs} respectively.

\subsection*{Notation} All logarithms are natural unless explicitly stated otherwise. 
We always consider $\eps$ to be a positive real number much smaller than $1$.
We use the standard $o(\cdot),\cO(\cdot)$ and $\Omega(\cdot)$ notation, and use $x(y) = \poly(y)$ as shorthand for \emph{there exists polynomials $P$ and $Q$ such that $P(y)< x(y) < Q(y)$ for all relevant $y$}. We denote the number of edges and vertices of a graph $G$ by $e(G)$ and $\abs{G}$, respectively. Moreover, we use $P_k$ to denote the path on $k$ vertices.

\subsection{Asymmetric graph removal lemma}

Generalising the triangle removal lemma, Erd\H{o}s, Frankl, and R\"{o}dl~\cite{EFR} proved the \emph{graph removal lemma}: if an $n$-vertex graph cannot be made $F$-free by deleting $\eps n^2$ edges, then it contains at least $\delta n^{\abs{F}}$ copies of $F$ where $\delta > 0$ depends only on $\eps$ and $F$.
The aforementioned bound of Ruzsa and Szemer\'{e}di shows that, for triangles and many other graphs, $\delta$ must be sub-polynomial in $\e$ (in fact, at most $\eps^{\Omega(\log(\eps^{-1}))}$).
There is particular interest in determining for which $F$ we have $\delta = \poly(\eps)$. 
If this is the case, then the removal lemma is said to be \defn{efficient} as it leads to algorithms for efficient property testing. 
Polynomial dependence in removal lemmas has been studied not only for subgraphs but also in the settings of arithmetic cycles \cite{fox2017tight,fox2018polynomial}, induced subgraphs~\cite{AlonFoxinduced,AlonShapirainduced}, digraphs~\cite{AlonShapiradirected}, ordered graphs~\cite{GTordered}, matrices~\cite{AlonBenEliezermatrix,AFNmatrix}, posets~\cite{poset}, and graph properties~\cite{GSproperties}.

Alon~\cite{Alondependence} proved that the only graphs for which $\delta = \poly(\eps)$ are bipartite graphs. Gishboliner, Shapira, and Wigderson~\cite{GSW23}, building on work of Csaba \cite{Csaba}, observed that the situation is more complicated for \defn{asymmetric} removal lemmas, where the graph that cannot be eliminated by deleting $\e n^2$ edges and the graph of which there are many copies are different. 
They showed that if an $n$-vertex graph cannot be made $C_{2k - 1}$-free by deleting $\eps n^2$ edges, then it contains at least $\poly(\eps) \cdot n^{2k + 1}$ copies of $C_{2k + 1}$. 
To capture such pairs of graphs with a polynomial dependency, Gishboliner, Shapira, and Wigderson defined a graph $H$ to be \defn{$F$-abundant} if any $n$-vertex graph that cannot be made $F$-free by deleting $\eps n^2$ edges must contain at least $\poly(\e) \cdot n^{\abs{H}}$ copies of $H$. 
Being $F$-abundant is a natural property that is preserved under taking subgraphs and blow-ups (see \cref{lem:abundant_facts}). The obvious question of course is which graphs are $F$-abundant?

Certainly we require that $H$ is homomorphic $F$, written $H \to F$ (see \cref{lem:abundant_facts}).
Since $F$-abundance is preserved under taking blow-ups and subgraphs, if $G$ is $F$-abundant and $H$ is homomorphic to $G$, then $H$ is $F$-abundant. 

Therefore, we can summarise the positive results of \cite{GSW23} as follows: for all $k\geq 2$, if $H$ is homomorphic to $C_{2k+1}$, then $H$ is $C_{2k-1}$-abundant. 
While this is a surprising result, odd cycles are a very particular family of graphs with many unusual properties.
Thus it is natural to wonder, given the results of \cite{GSW23}, if abundance is just another quirk of odd cycles.
Indeed, Gishboliner, Shapira, and Wigderson asked \cite[Problem~1.9]{GSW23} whether there are any $K_t$-abundant graphs of chromatic number $t$ for $t \geq 4$ and mentioned that they were inclined to believe that the answer was no. 
We show that the answer to the question  is \emph{yes} in a strong sense: $F$-abundant graphs have no restriction whatsoever on their homomorphic images beyond the trivial requirement that they are homomorphic to $F$. 

\begin{theorem}\label{thm:FabundantchiF}
    Let $F$ and $G$ be graphs with $F \nrightarrow G$. There is an $F$-abundant graph $H$ with $H \nrightarrow G$.
\end{theorem}

In particular, taking $G$ to be the complete graph on $\chi(F) - 1$ vertices shows that for any graph $F$ there is an $F$-abundant graph with the same chromatic number as $F$.
By standard arguments one can also take  $H$  to have arbitrarily high girth (see \cref{cor:abundant_girth}). 
Our method of proof for \cref{thm:FabundantchiF} employs a number of glueing type operations (see \cref{lem:peel_abundant,lem:glue_blowup_abundant})
on graphs that we show maintain $F$-abundance. 
This gives lots of flexibility when constructing $F$-abundant graphs,
and shows that the family of pairs of graphs that satisfy an abundance relation is very rich indeed (see \cref{thm:splittable_abundant}). 

For the simplest abundant pair, which is $C_5$ and $K_3$, we also demonstrate how to prove effective bounds on the degree of the polynomial dependency. 
This analysis also extends to the sparse case, where $\eps$ may depend on $n$, see \cref{sec:effective}.

\subsection{Avoiding solutions to equations in the integers}

Consider the equations 
\begin{equation*}
    x + y = 2z \quad \text{ and } \quad x + y = z + w.
\end{equation*}
An integer solution to the first is a $3$-term arithmetic progression, while the second is the famed Sidon equation~\cite{Sidon}. 
It is known that any set $\cA\subset [N]$ that contains no non-trivial solution to the Sidon equation has size at most $\cO(\sqrt{N})$ while there exist sets with no $3$-term arithmetic progressions of size $N^{1-o(1)}$.

In order to understand the source of this wildly differing behaviour, Ruzsa~\cite{Ruzsa} initiated the systematic study of \defn{translation-invariant equations}, namely equations
\begin{equation}\label{eq:translationinvariant}
    a_1 x_1 + a_2 x_2 + \dots + a_k x_k = 0
\end{equation}
with non-zero integer coefficients satisfying $a_1 + \dots + a_k = 0$.
Such equations always have trivial solutions where all variables are equal. For a translation-invariant equation $E$, Ruzsa defined the two quantities (see \cref{sec:additive} for the definitions of non-trivial, genus, and symmetric in what follows)
\begin{align*}
    r_E(N) & = \max\set{\abs{\cA} : \cA \subset [N], \text{there is no non-trivial solution to $E$ with the $x_i$ all in $\cA$}}, \\
    R_E(N) & = \max\set{\abs{\cA} : \cA \subset [N], \text{there is no solution to $E$ with the $x_i$ all distinct and in $\cA$}}.
\end{align*}
For any equation $E$, $r_E(N) \leq R_E(N) \leq N (\log N)^{-\alpha_E}$ for some positive constant $\alpha_E$, but this upper bound may be far from tight as in the case of the Sidon equation. 
In fact, Ruzsa showed that if $E$ has genus $m$, then $r_E(N) = \cO(N^{1/m})$.
In the special case that $E$ is a symmetric equation in at least $4$ variables (and so has genus at least two), Ruzsa showed that $R_E(N) = \cO(\sqrt{N})$. However, for a general equation $E$ of genus at least two it was not known whether there is any such improvement for the upper bound on $R_E$. We show that there is.
\begin{theorem}\label{thm:genus2sqrt}
    Let $E$ be an equation of genus at least $2$. Then $R_E(N) = \cO(\sqrt{N})$.
\end{theorem}

In fact we also show that there are many such solutions, see \cref{cor:abundance}.
We note that for any integer $\ell \geq 2$ and any $\eps > 0$, Ruzsa showed that there is an equation $E$ of genus $\ell$ for which $R_E(N) = \Omega(N^{1/2 - \eps})$.
Thus the exponent in \cref{thm:genus2sqrt} is the best one can hope for.

\subsection{From graphs to equations and back}

An $F$-coloured graph $(G,\sigma)$ is a graph $G$ along with a homomorphism $\sigma \colon G \to F$, and one can analogously define $F$-abundance for $F$-coloured graphs (see \cref{sec:unifFabundant}).
Ruzsa and Szemer\'{e}di~\cite{RS78} provided a way of associating a collection of systems of equations $\cE(G, \sigma)$ to an  $F$-coloured graph, with each system $S \in \cE(G, \sigma)$ containing one equation for each cycle of the graph.
This translated the problem of proving non-$F$-abundance of an $F$-coloured-graph to the problem of proving that there exists a set $\cA \subset [N]$ of size $N^{1 - o(1)}$ containing few solutions (with distinct integers) to any of the systems of equations. 
We thus extend the definition of $R_E(N)$ to systems of equations in the obvious way, where a solution to a system is simultaneously a solution to all of the equations. 
Further, we say a system of equations $S$ is \defn{avoidable} if  $R_S(N) = N^{1 - o(1)}$.
In this language, the Ruzsa-Szemer\'{e}di construction shows:
\begin{quote}
    \textit{If a coloured graph $(G, \sigma)$ is $F$-abundant, then every system $S \in \cE(G, \sigma)$ is unavoidable.} 
\end{quote}
The only method we have for proving non-abundance is this construction of Ruzsa and Szemer\'{e}di, and the only equations that are known to be avoidable are \defn{convex} equations~\cite[Thm.~2.3]{Ruzsa}, in which all but one of the coefficients in \eqref{eq:translationinvariant} have the same sign, such as $x + y = 2z$.
\Cref{thm:genus2sqrt} says that equations of genus at least 2 are unavoidable.
This leaves open the case of equations of genus $1$ that are not convex. 
Indeed, Ruzsa remarked that the following\footnote{In fact he asked a stronger question but \cite{Bukh} proved this to be false.}
remained a possibility, although there was ``too little evidence to call it a conjecture''.
\begin{question}[{\cite[(9.1)]{Ruzsa}}]\label{qu:Ruzsa}
    Are all genus $1$ equations avoidable?
\end{question}
The answer to Ruzsa's question has immediate implications for graph removal lemmas. 
For example, as far as  we know, it may be that all tripartite triangle-free graphs are $K_3$-abundant. However, Gishboliner, Shapira, and Wigderson~\cite{GSW23} showed that \emph{if the answer to Ruzsa's question is yes}, then non-abundant tripartite triangle-free graphs exist.
Indeed, they showed that an appropriately chosen random graph would bear witness. 
Nevertheless, they also asked whether one could \emph{unconditionally} show that a non-abundant tripartite triangle-free graph exists. 
They suggested a way to do this via convex equations.
We make two contributions here. 
We show that the proposed method to use convex equations cannot work (see \cref{lem:convex_not_possible}) and we provide a small explicit example of a tripartite triangle-free graph that is non-abundant \emph{if} Ruzsa's question has a positive answer. 

Both within this paper and outside of it, the connections between equations and graphs have thus far been one way, using avoidable equations to prove non-abundance of graphs. 
It therefore remained possible that the additive angle is simply a convenient way of proving non-abundance and nothing more. 
We show that the connection between the two questions is deeper than that.
As some setup is required we only informally state our two results here (see \cref{sec:eqtographs} for details).
Firstly, we show that in the case of coloured cycles, a family of graphs containing both abundant and non-abundant graphs, we have a complete picture (see \cref{thm:cycleabundance}) which can be paraphrased as:
\begin{quote}
    For an $F$-coloured cycle $(C,\sigma)$, the avoidability of $\cE(C,\sigma)$ determines the $F$-abundance of $(C,\sigma)$.
\end{quote}
Of course, cycles are a particularly simple family, but we also prove a result for general $K_3$-coloured graphs (see \cref{lem:symmetric_is_path_hom}) which can be paraphrased as:
\begin{quote}
    For a $K_3$-coloured graph $(G, \sigma)$, if $\cE(G, \sigma)$ is symmetric (in Ruzsa's sense), then $G$ is triangle-abundant. 
\end{quote}
In fact we show a stronger if and only if statement, but that requires some setup.

\section{\texorpdfstring{$F$}{F}-abundant graphs}\label{sec:Abundance}

We start by formally introducing abundance, giving some of its basic properties, and by developing useful machinery. 
With this in hand we prove that certain glueing type operations (see \cref{lem:peel_abundant,lem:glue_blowup_abundant}) maintain abundance which allows us to construct a large family of $F$-abundant graphs (see \cref{thm:splittable_abundant}). 
We use this to resolve some open cases such as showing that the Petersen graph is triangle-abundant and prove the main result of this section, \cref{thm:FabundantchiF}.

We start with some useful definitions. An $n$-vertex graph is \defn{$\eps$-far from $F$-free} if it cannot be made $F$-free by deleting fewer than $\eps n^2$ edges. $F$-abundance can now be phrased as follows.

\begin{definition}[$F$-abundant]\label{def:Fabundance}
    Let $H$ and $F$ be graphs. $H$ is \defn{$F$-abundant} if for all $\eps > 0$ and for all sufficiently large\footnote{The `sufficiently large' condition was not stated explicitly in \cite{GSW23} but was implicitly used there.} $n$, every $n$-vertex graph that is $\eps$-far from $F$-free contains $\poly(\eps) \cdot n^{\abs{H}}$ copies of $H$. The polynomial may only depend on $H$ and $F$.
\end{definition}

There is a simple condition that $H$ must satisfy to have any chance of being $F$-abundant. A \defn{blow-up} of a graph $F$ is obtained by replacing each vertex $i$ of $F$ by a non-empty independent set $V_i$ and each edge $ij$ by a complete bipartite graph between parts $V_i$ and $V_j$. Blow-ups of $F$ where each $V_i$ has the same size are $\eps$-far from $F$-free and so, for $H$ to have any chance of being $F$-abundant, $H$ must be a subgraph of a blow-up of $F$. This last condition is very natural in the language of graph homomorphisms.

\begin{definition}[Homomorphisms]\label{def:Fpartite}
    A graph $G$ is \defn{$F$-partite} if there is a partition $(V_i : i \in V(F))$ of $V(G)$ into independent sets such that if $uv \in E(G)$ for $u \in V_i$ and $v \in V_j$, then $ij \in E(F)$. The partition $(V_i : i \in V(F))$ is an \defn{$F$-partition} of $G$.

    The map $\sigma \colon G \to F$ defined so that $v \in V_{\sigma(v)}$ for every vertex of $G$ is called an \defn{$F$-colouring} of $G$ or a \defn{homomorphism} from $G$ to $F$. If $G$ has an $F$-colouring, then we say that $G$ is \defn{homomorphic} to $F$ and denote this by \defn{$G \to F$}.
\end{definition}

We now collect together some basic facts about $F$-abundance before developing some machinery in the next section.

\begin{lemma}\label{lem:abundant_facts}
    Let $H$ and $F$ be graphs.
    \begin{enumerate}[label = \arabic{*}.]
        \item If $H$ is $F$-abundant, then $H \to F$.
        \item $F$ is $F$-abundant if and only if $F$ is bipartite.
        \item $F$-abundance is preserved under taking subgraphs and blow-ups. In particular, if $H$ is $F$-abundant and $H' \to H$, then $H'$ is $F$-abundant.
        \item If $e(F) > 0$, then every bipartite graph is $F$-abundant. 
    \end{enumerate}
\end{lemma}

\begin{proof}
    Note that $H$ is $F$-partite if and only if $H$ is a subgraph of some blow-up of $F$. This together with the discussion preceding \cref{def:Fpartite} proves part 1. Part 2 is just a restatement of Alon's result~\cite{Alondependence} that there is polynomial dependence in the $F$-removal lemma if and only if $F$ is bipartite.

    Suppose an $n$-vertex graph $G$ contains $\eps n^{\abs{H}}$ copies of $H$. Let $H'$ be a subgraph of $H$. Each copy of $H'$ in $G$ is contained in at most $n^{\abs{H} - \abs{H'}}$ copies of $H$ and so $G$ contains at least $\eps n^{\abs{H'}}$ copies of $H'$. Now let $H''$ be a blow-up of $H$. A standard supersaturation argument (such as \cite[Lem.~2.1]{Keevash2011hypergraphsurvey}) together with Erd\H{o}s's result on the extremal function for complete $\ell$-uniform $\ell$-partite hypergraphs~\cite{Erdos1964hypextremal} shows that $G$ contains $\poly(\eps) \cdot n^{\abs{H''}}$ copies of $H''$. Thus $F$-abundance is preserved under taking subgraphs and blow-ups which proves part 3.

    Finally, if $e(F) > 0$, then a single edge is plainly $F$-abundant. Every bipartite graph is homomorphic to an edge and so is $F$-abundant by part 3.
\end{proof}

\subsection{Canonical structure of graphs that are \texorpdfstring{$\e$}{eps}-far from \texorpdfstring{$F$}{F}-free}\label{sec:unifFabundant}

To understand $F$-abundance it will first be useful to understand the structure of graphs $G$ that are $\eps$-far from $F$-free. We will show that we may pass to an $F$-partite subgraph of $G$ that is $\Omega(\eps)$-far from $F$-free and has some additional structure that will be helpful for proving $F$-abundance.

We say that a graph $G$ is \defn{uniformly $\e$-far from $F$-free} if it has an $F$-partition $(V_i : i \in V(F))$ and a collection $\cC$ of edge-disjoint copies of $F$ such that
\begin{itemize}
    \item for every copy of $F$ in $\cC$ and each $i \in V(F)$, the vertex corresponding to $i$ is in $V_i$, and
    \item every vertex of $G$ is in at least $\e \cdot \abs{G}$ copies of $F$ from $\cC$.
\end{itemize}
We first note a few simple properties of graphs that are uniformly $\e$-far from $F$-free.

\begin{lemma}\label{lem:euniffacts}
    Let $G$ be uniformly $\e$-far from $F$-free and let $(V_i : i \in V(F))$ be the corresponding $F$-partition of $G$. Then,
    \begin{itemize}
        \item $G$ is $\e/\abs{F}$-far from $F$-free,
        \item for all $i \in V(F)$, every vertex of $V_i$ has at least $\e \cdot \abs{G}$ neighbours in $V_j$ for every neighbour $j$ of $i$ in $F$, and
        \item $\abs{V_i} \geq \e \cdot \abs{G}$ for every non-isolated vertex $i$ of $F$.
    \end{itemize}
\end{lemma}

\begin{proof}
    Let $V_i$ be the largest part in the $F$-partition, so $\abs{V_i} \geq \abs{G}/\abs{F}$. Each vertex of $V_i$ is in at least $\e \cdot \abs{G}$ copies of $F$ from $\cC$ which implies that
    \begin{equation*}
        \abs{\cC} \geq \abs{V_i} \cdot \e \cdot \abs{G} \geq \e/\abs{F} \cdot \abs{G}^2.
    \end{equation*}
    To make $G$ $F$-free requires deleting an edge from each copy of $F$ in $\cC$, and these copies are edge-disjoint. Hence, at least $\abs{\cC}$ edges need deleting and so $G$ is $\e/\abs{F}$-far from $F$-free.

    Fix two adjacent vertices $i$ and $j$ of $F$ and let $v \in V_i$. There are at least $\e \cdot \abs{G}$ copies of $F$ in $\cC$ that contain $v$ and every such copy also contains a vertex $u$ in $V_j$ that is adjacent to $v$. As the copies of $F$ in $\cC$ are edge-disjoint, the vertices $u$ in $V_j$ are all distinct and so $v$ has at least $\e \cdot \abs{G}$ neighbours in $V_j$.
    
    Finally, let $i$ be a non-isolated vertex of $F$. Let $j$ be a neighbour of $i$ in $F$ and fix some $v \in V_j$. Vertex $v$ has at least $\e \cdot \abs{G}$ neighbours in $V_i$ and so $\abs{V_i} \geq \e \cdot \abs{G}$.
\end{proof}

\Cref{lem:euniffacts} shows that graphs which are uniformly far from $F$-free are far from $F$-free (up to a multiplicative loss of $\abs{F}$). Next we prove a converse: every graph that is far from $F$-free contains a large subgraph that is uniformly far from $F$-free. In particular, graphs which are uniformly $\e$-far from $F$-free are canonical for graphs that are $\e$-far from $F$-free with only a constant multiplicative loss in the parameter $\e$, and so we can use such graphs for showing $F$-abundance.

\begin{lemma}\label{lem:eunifsubgraph}
    Let $G$ be an $n$-vertex graph that is $\e$-far from $F$-free. Then $G$ has a subgraph $G'$ of order at least $(\e/\abs{F}^{\abs{F}})^{1/2} \cdot n$ that is uniformly $\e/(2 \cdot \abs{F}^{\abs{F}})$-far from $F$-free.
\end{lemma}

\begin{proof}
    Since $G$ is $\e$-far from $F$-free, it contains a collection $\cC$ of $\e n^2$ edge-disjoint copies of $F$. First take a random $F$-partition of $G$ and say a $C \in \cC$ \defn{survives} if the vertices of $C$ are in the corresponding parts of the $F$-partition. Each $C \in \cC$ survives with probability $1/\abs{F}^{\abs{F}}$ so, for some $F$-partition, at least $1/\abs{F}^{\abs{F}}$ of copies of $F$ in $\cC$ survive. Let $G_1$ be the corresponding $F$-partite subgraph and $\cC_1$ be the set of surviving copies of $F$. Note that $\abs{\cC_1} \geq \abs{\cC}/\abs{F}^{\abs{F}} \geq \e/\abs{F}^{\abs{F}} \cdot n^2$.

    Next we sparsify. If $v \in G_1$ is in fewer than $\e/(2 \cdot \abs{F}^{\abs{F}}) \cdot n$ copies of $F$ from $\cC_1$, then delete the edges of all these copies of $F$. Do this until every vertex is either in no copy of $F$ from $\cC_1$ or in at least $\e/(2 \cdot \abs{F}^{\abs{F}}) \cdot n$ of them. The number of deletions is less than $\e/(2 \cdot \abs{F}^{\abs{F}}) \cdot n^2$ and so at least $\e/(2 \cdot \abs{F}^{\abs{F}}) \cdot n^2$ copies of $F$ from $\cC_1$ remain. Call the set of remaining copies of $F$ $\cC'$. Let $G'$ be the graph with vertex-set $\bigcup_{C \in \cC'} V(C)$ and edge-set $\bigcup_{C \in \cC'} E(C)$.

    Note that $G'$ is an $F$-partite subgraph of $G$, $\cC' \subset \cC$ is a collection of edge-disjoint copies of $F$ in $G'$, and every vertex of $G'$ is in at least $\e/(2 \cdot \abs{F}^{\abs{F}}) \cdot n$ copies of $F$ from $\cC'$. Hence, $G'$ is uniformly $\e/(2 \cdot \abs{F}^{\abs{F}})$-far from $F$-free. Finally, since $\abs{\cC'} \geq \e/(2 \cdot \abs{F}^{\abs{F}}) \cdot n^2$, $G'$ has at least that many edges and so has order at least $(\e/\abs{F}^{\abs{F}})^{1/2} \cdot n$.
\end{proof}

\Cref{lem:euniffacts,lem:eunifsubgraph} give the canonical $F$-partite structure of graphs that are $\e$-far from $F$-free. With this in mind, it is natural to talk about \defn{$F$-coloured graphs} being $F$-abundant. An $F$-coloured graph is a pair $(H, \sigma)$ where $H$ is a graph and $\sigma$ is an $F$-colouring of $H$.

\begin{definition}
    Let $(H, \sigma)$ be an $F$-coloured graph. $(H, \sigma)$ is \defn{$F$-abundant} if for all $\eps > 0$ and for all sufficiently large $n$, every $n$-vertex graph that is uniformly $\eps$-far from $F$-free graph with corresponding $F$-partition $(V_i : i \in V(F))$ contains $\poly(\eps) \cdot n^{\abs{H}}$ copies of $(H, \sigma)$ (that is, each $v \in H$ is embedded in part $V_{\sigma(v)}$). The polynomial may only depend on $F$, $H$, and $\sigma$.
\end{definition}

To show that a graph $H$ is $F$-abundant, it suffices to show that there is some $F$-colouring $\sigma$ of $H$ such that $(H, \sigma)$ is $F$-abundant. In fact, all known proofs of abundance implicitly do this.
Whether or not this is necessary is an important question, related to the compactness of abundance, see \cref{sec:questions}.

\subsection{Building up abundance} Now we develop some operations that allow us to build $F$-abundant $F$-coloured graphs from smaller ones. The first lets us add vertices with monochromatic neighbourhoods.

\begin{theorem}\label{lem:peel_abundant}
    Let $(H, \sigma)$ be an $F$-coloured graph that is $F$-abundant. Let $v_1, \dots, v_s \in V(H)$ be vertices satisfying $\sigma(v_1) = \sigma(v_2) = \dotsb = \sigma(v_s)$. Let $(H', \sigma')$ be an $F$-coloured graph obtained from $(H, \sigma)$ by joining a new vertex $u$ to $v_1, \dots, v_s$ and taking $\sigma'(u)$ to be any neighbour of $\sigma(v_1)$ in $F$. Then $(H', \sigma')$ is $F$-abundant.
\end{theorem}

\begin{proof}
    Without loss of generality, we may assume that $V(F) = \set{1, 2, \dotsc, \abs{F}}$, $\set{1,2}$ is an edge of $F$, $\sigma(v_1) = \sigma(v_2) = \dotsb = \sigma(v_s) = 1$, and $\sigma'(u) = 2$. Let $G$ be an $n$-vertex graph that is uniformly $\e$-far from $F$-free. Let $(V_i : i \in V(F))$ be the corresponding $F$-partition of $G$, and let $\cC$ be the corresponding collection of edge-disjoint copies of $F$.
    
    Fix any vertex $x \in V_2$ (this will play the role of $u$) and let $U_1 = N(x) \cap V_1$ (we will find the vertices corresponding to $v_1, \dots, v_s$ in $U_1$). Note that, by \cref{lem:euniffacts}, $\abs{U_1} \geq \e n$. Consider $G_1 = G[U_1, V_2, V_3, \dotsc, V_{\abs{F}}] - x$ and let $\cC_1 = \set{C \in \cC : V(C) \cap U_1 \neq \emptyset \text{ and } x \notin V(C)}$. Since every vertex of $U_1$ is in at least $\e n$ copies of $F$ from $\cC$ and $x$ is in at most $n$ copies of $F$ from $\cC$ (the copies of $F$ from $\cC$ are edge-disjoint), $\cC_1$ is a collection of at least $\Omega(\e^2) \cdot n^2$ edge-disjoint copies of $F$ in $G_1$.

    By \cref{lem:eunifsubgraph}, there is an uniformly $\Omega(\e^2)$-far from $F$-free subgraph $G'$ of $G_1$ of order $\Omega(\e) \cdot \abs{G_1} = \Omega(\e^2) \cdot n$. Let $A_1 \subset U_1$, $A_2 \subset V_2$, \ldots, $A_{\abs{F}} \subset V_{\abs{F}}$ be the corresponding $F$-partition of $G'$. Since $(H, \sigma)$ is $F$-abundant, $G'$ contains $\poly_1(\e) \cdot n^{\abs{H}}$ copies of $(H, \sigma)$. In particular, each $v_i$ is embedded into $A_1 \subset U_1 \subset N(x)$. Thus, adding $x$ to any such copy of $(H, \sigma)$ gives a copy of $(H', \sigma')$. Since there were $\abs{V_2} \geq \e n$ choices for $x$ in $V_2$, there are $\poly_2(\e) \cdot n^{\abs{H'}}$ copies of $(H', \sigma')$ in $G$.
\end{proof}

\begin{example}
    One can see that $C_5$ is $K_3$-abundant by noting that the $3$-colouring of $C_5$ in the middle of \cref{fig:c5} has a vertex with a monochromatic neighbourhood (in fact, any 3-colouring has this property). 
    Once this vertex has been deleted the resulting graph (depicted on the left of the figure) is a path, which is abundant (ignore the right hand picture for now).

    \begin{figure}[ht!]
        \centering
        \begin{tikzpicture}
            \def\pstyle{{"circ","square","triangle","circ","triangle"}}
            \draw[->] (2.5,0) -- (3.5,0);
            \draw[->] (9.5,0) -- (8.5,0);
            \begin{scope}[shift={(0,0)}]
                \foreach \i in {0,1,2,4} {
                    \pgfmathparse{\pstyle[\i]};
                    \node[\pgfmathresult] (P\i) at (72*\i+18:1) {};
                }
                \draw (P4) -- (P0) -- (P1) -- (P2);
            \end{scope}
            \begin{scope}[shift={(6,0)}]
                \foreach \i in {0,...,4} {
                    \pgfmathparse{\pstyle[\i]};
                    \node[\pgfmathresult] (P\i) at (72*\i+18:1) {};
                }
                \draw (P0) -- (P1) -- (P2) -- (P3) -- (P4) -- (P0);
            \end{scope}
            \begin{scope}[shift={(12,-0.65)}]
                \foreach \i in {0,2,3,4} {
                    \pgfmathparse{\pstyle[\i]};
                    \node[\pgfmathresult] (P\i) at (72*\i+18:1) {};
                }
                \foreach \i in {0,1,2} {
                    \pgfmathparse{\pstyle[\i]};
                    \node[\pgfmathresult] (Q\i) at ($(72*\i+18:1)+(0,1.3)$) {};
                }
                \draw (Q0) -- (Q1) -- (Q2);
                \draw (P0) -- (P2) -- (P3) -- (P4) -- (P0);
                \draw[<->] ($(P0)+(0,0.4)$) -- ($(Q0)-(0,0.4)$);
                \draw[<->] ($(P2)+(0,0.4)$) -- ($(Q2)-(0,0.4)$);
            \end{scope}
        \end{tikzpicture}
        \caption{Two ways to show the $K_3$-abundance of $C_5$.} \label{fig:c5}
    \end{figure}
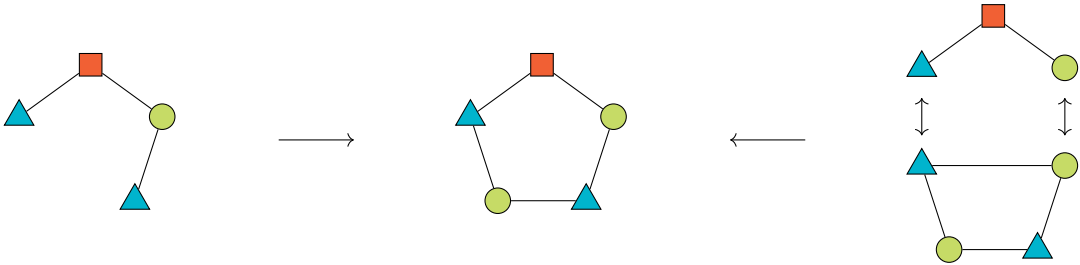
\end{example}

The second operation allows us to glue along an edge.

\begin{theorem}\label{lem:glue_blowup_abundant}
    Let $(H, \sigma)$ and $(H', \sigma')$ be $F$-coloured graphs that are $F$-abundant. Let $uv \in E(H)$,
    $U',V' \subset V(H')$  satisfy $\sigma'(U') = \sigma(u)$, and $\sigma'(V') = \sigma(v)$.
    Let the $F$-coloured graph $(H'',\sigma'')$ be obtained by 
    \begin{itemize}
        \item deleting the edge $uv$ from $H$,
        \item blowing up $u$ by $\abs{U'}$ and $v$ by $\abs{V'}$,
        \item identifying each vertex of $U'$ with a unique copy of $u$ and each vertex of $V'$ with a unique copy of $v$.  
    \end{itemize}
    Then $(H'', \sigma'')$ is $F$-abundant.
\end{theorem}

The following example is particularly simple and worth bearing in mind for the proof. 
Let $H$ be the $C_4$ and $H'$ be the path.

\begin{example}
    The $K_3$-coloured $C_4$ on the right of \cref{fig:c5} is $K_3$-abundant (for example, by \cref{lem:peel_abundant}). 
    Moreover, the path depicted on the right is also $K_3$-abundant. Thus, by applying \cref{lem:glue_blowup_abundant} we see that the $K_3$-coloured $C_5$ in the middle of \cref{fig:c5} is $K_3$-abundant.
\end{example}

\begin{proof}[Proof of \cref{lem:glue_blowup_abundant}]
    Without loss of generality, we may assume that $V(F) = \set{1, 2, \dotsc, \abs{F}}$, $\set{1, 2}$ is an edge of $F$, $\sigma(u) = \sigma'(U') = 1$, and $\sigma(v) = \sigma(V') = 2$. Let $G$ be an $n$-vertex graph that is uniformly $\e$-far from $F$-free. Let $(V_i : i \in V(F))$ be the corresponding $F$-partition of $G$, and let $\cC$ be the corresponding collection of edge-disjoint copies of $F$. By deleting edges, we may assume that every edge of $G$ is in a copy of $F$ from $\cC$.

    Since $(H, \sigma)$ is $F$-abundant, there are $\poly_1(\e) \cdot n^{\abs{H}}$ copies of $(H, \sigma)$ in $G$. Let $H_1 \coloneqq H - u - v$. The \defn{weight} of a copy of $(H_1, \sigma)$ in $G$ is the number of $(H, \sigma)$ that contain it. This is just the number of $(x, y) \in V_1 \times V_2$ which when added to the copy of $(H_1, \sigma)$ give a copy of $(H, \sigma)$. The sum of the weights of the copies of $(H_1, \sigma)$ is at least $\poly_1(\e) \cdot n^{\abs{H}}$, the number of copies of $(H_1, \sigma)$ is at most $n^{\abs{H_1}} = n^{\abs{H} - 2}$, and the maximum weight of any single copy of $(H_1, \sigma)$ is at most $n^2$. Hence, there are at least $\poly_2(\e) \cdot n^{\abs{H} - 2}$ copies of $(H_1, \sigma)$ that have weight at least $\poly_3(\e) \cdot n^2$.

    Fix a copy $(H_1, \sigma)$ of weight at least $\poly_3(\e) \cdot n^2$ and call it $\cH$. Consider the pairs $(x, y) \in V_1 \times V_2$ that extend $\cH$ to a copy of $(H, \sigma)$. Since $uv \in E(H)$, each such pair is an edge in $G$. So if $X \subset V_1$ and $Y \subset V_2$ are the end-vertices of these edges, there are $\poly_3(\e) \cdot n^2$ edges from $X$ to $Y$. Crucially, every pair $(x, y)$ with $x \in X$ and $y \in Y$ extends $\cH$ to a copy of $(H - uv, \sigma)$. 
    
    Let $\cC_1 = \set{C \in \cC : C \cap X \neq \emptyset \text{ and } C \cap Y \neq \emptyset}$. Since there are $\poly_3(\e) \cdot n^2$ edges from $X$ to $Y$ and every edge of $G$ is in a copy of $F$ from $\cC$, there are $\poly_3(\e) \cdot n^2$ edge-disjoint copies of $F$ in $\cC_1$. Let $G_1$ be the subgraph of $G$ induced by $\cC_1$. By \cref{lem:eunifsubgraph}, $G_1$ has a subgraph $G'$ of order at least $\poly_4(\e) \cdot n$ that is uniformly $\poly_5(\e)$-far from $F$-free. Furthermore, the corresponding $F$-partition $(V'_i : i \in V(F))$ of $G'$ satisfies $V'_1 \subset X$, $V'_2 \subset Y$, $V'_3 \subset V_3$, \ldots, $V'_{\abs{F}} \subset V_{\abs{F}}$. Since $(H', \sigma')$ is $F$-abundant, $G'$ contains $\poly_6(\e) \cdot n^{\abs{H'}}$ copies of $(H', \sigma')$ and at most $\abs{H} \cdot n^{\abs{H'} - 1}$ of these intersect $\cH$. Each of the remaining copies extend $\cH$ to a copy of $(H'', \sigma'')$, since $U'$ embeds into $V'_1 \subset X$ and $V'$ embeds into $V'_2 \subset Y$.

    Thus $\cH$ extends to at least $\poly_7(\e) \cdot n^{\abs{H'}}$ copies of $(H'', \sigma'')$. There were $\poly_2(\e) \cdot n^{\abs{H} - 2}$ choices for $\cH$ and so $G$ contains at least $\poly_8(\e) \cdot n^{\abs{H'} + \abs{H} - 2} = \poly_8(\e) \cdot n^{\abs{H''}}$ copies of $(H'', \sigma'')$. Thus $(H'', \sigma'')$ is $F$-abundant.
\end{proof}

\begin{example}
    The $K_3$-coloured Petersen graph depicted on the right in \cref{fig:petersen} can be seen to be $K_3$-abundant by a combination of \cref{lem:peel_abundant} and \cref{lem:glue_blowup_abundant} as indicated in the figure. The left graph (without the square vertices) is 2-coloured and therefore $K_3$-abundant. Then, we add the square vertices to that graph according to \cref{lem:glue_blowup_abundant} by replacing one edge each. Thus, the middle graph (without the circular top vertex) is abundant. Finally, we add the circular top vertex to that graph. The resulting graph, which is a coloured Petersen graph, is then abundant by \cref{lem:peel_abundant}. 
    Notably, neither \cref{lem:peel_abundant} nor \cref{lem:glue_blowup_abundant} suffice alone.

    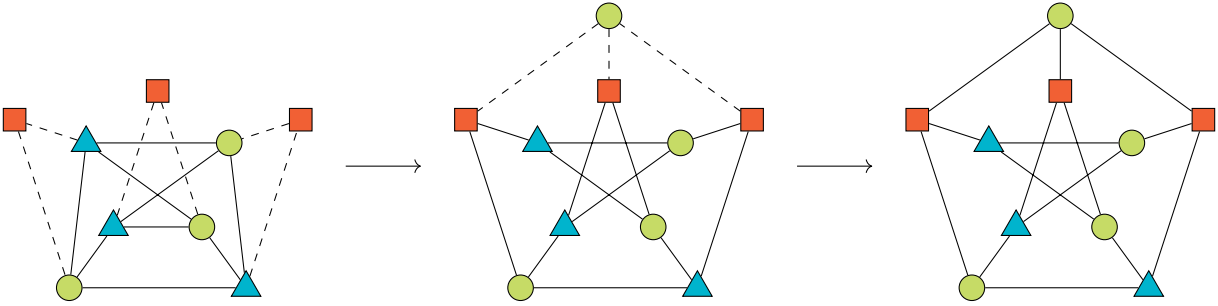
\begin{figure}[ht!]
        \centering
        \begin{tikzpicture}
            \def\pstyle{{"square","circ","square","circ","triangle"}}
            \def\qstyle{{"circ","square","triangle","triangle","circ"}}
            \draw[->] (2.5,0) -- (3.5,0);
            \draw[->] (8.5,0) -- (9.5,0);
            \begin{scope}
                \foreach \i in {0,2,3,4} {
                    \pgfmathparse{\pstyle[\i]};
                    \node[\pgfmathresult] (P\i) at (72*\i+18:2) {};
                }
                \foreach \i in {0,...,4} {
                    \pgfmathparse{\qstyle[\i]};
                    \node[\pgfmathresult] (Q\i) at (72*\i+18:1) {};
                }
                \draw[dashed] (Q2) -- (P2) -- (P3) (P4) -- (P0) -- (Q0) (Q4) -- (Q1) -- (Q3);
                \draw (Q2) -- (P3) edge (Q3) -- (P4) edge (Q4) -- (Q0);
                \draw (Q0) -- (Q2) -- (Q4) -- (Q3) -- (Q0);
            \end{scope}
            \begin{scope}[shift={(6,0)}]
                \foreach \i in {0,...,4} {
                    \pgfmathparse{\pstyle[\i]};
                    \node[\pgfmathresult] (P\i) at (72*\i+18:2) {};
                    \pgfmathparse{\qstyle[\i]};
                    \node[\pgfmathresult] (Q\i) at (72*\i+18:1) {};
                }
                \draw[dashed] (P0) -- (P1) edge (Q1) -- (P2);
                \draw (P0) edge (Q0) (P2) edge (Q2) -- (P3) edge (Q3) -- (P4) edge (Q4) -- (P0);
                \draw (Q0) -- (Q2) -- (Q4) -- (Q1) -- (Q3) -- (Q0);
            \end{scope}
            \begin{scope}[shift={(12,0)}]
                \foreach \i in {0,...,4} {
                    \pgfmathparse{\pstyle[\i]};
                    \node[\pgfmathresult] (P\i) at (72*\i+18:2) {};
                    \pgfmathparse{\qstyle[\i]};
                    \node[\pgfmathresult] (Q\i) at (72*\i+18:1) {};
                }
                \draw (P0) edge (Q0) -- (P1) edge (Q1) -- (P2) edge (Q2) -- (P3) edge (Q3) -- (P4) edge (Q4) -- (P0);
                \draw (Q0) -- (Q2) -- (Q4) -- (Q1) -- (Q3) -- (Q0);
            \end{scope}
        \end{tikzpicture}
        \caption{$K_3$-abundance of the Petersen graph.} \label{fig:petersen}
    \end{figure}
\end{example}

Our third operation is a particularly useful combination of \cref{lem:peel_abundant} and \cref{lem:glue_blowup_abundant}.

\begin{theorem}\label{lem:join_abundant}
    Let $(H, \sigma)$ and $(H', \sigma')$ be $F$-coloured graphs that are $F$-abundant. Let $a, b, a'$ be vertices of $F$ with $ab, aa' \in E(F)$. The coloured graph obtained from the disjoint union of $(H, \sigma)$ and $(H', \sigma')$ by adding all edges between vertices of colour $a$ in $H$ and colour $a'$ in $H'$ and all edges between vertices of colour $b$ in $H$ and colour $a$ in $H'$ is $F$-abundant.
\end{theorem}

\begin{proof}
    Let $(H_1, \sigma_1)$ be the coloured graph obtained from $(H, \sigma)$ by adding two vertices $u$ and $v$ where:
    \begin{itemize}[noitemsep]
        \item $u$ is joined to $v$,
        \item $u$ is joined to all vertices of colour $a$ in $H$,
        \item $v$ is joined to all vertices of colour $b$ in $H$,
        \item $\sigma(u) = a'$, and
        \item $\sigma(v) = a$.
    \end{itemize}
    Applying \cref{lem:peel_abundant} to $(H, \sigma)$ and adding the vertices $v$ and $u$ in that order shows that $(H_1, \sigma_1)$ is $F$-abundant. Let $U' = \set{x \in H' : \sigma'(x) = a'}$ and $V' = \set{x \in H' : \sigma'(x) = a}$. Applying \cref{lem:glue_blowup_abundant} to $(H_1, \sigma_1)$ and $(H', \sigma')$ gives the claimed result.
\end{proof}

The case $a' = b$ in \cref{lem:join_abundant} is particularly clean, and yields the following family.

\begin{definition}
    A graph $H$ is \defn{$F$-splittable} if there is an $F$-colouring $\sigma$ of $H$ with the following property. Every subgraph $H'$ of $H$ has an edge cut-set whose edges span only two colours.
\end{definition}

From the previous theorem it follows that all splittable graphs are abundant.

\begin{theorem}\label{thm:splittable_abundant}
    Every $F$-splittable graph is $F$-abundant.
\end{theorem}

\begin{proof}
    Let $H$ be $F$-splittable and $\sigma$ the $F$-colouring that witnesses this. We prove that $(H', \sigma\vert_{V(H')})$ is $F$-abundant for every subgraph $H'$ of $H$. We do this by induction on the number of edges. The result holds for all subgraphs with no edges.

    Let $H'$ be a subgraph of $H$ with at least one edge. As $H$ is $F$-splittable with witness $\sigma$, there are two colours $a$ and $b$ such that the set $E_{ab}$ of edges in $H'$ spanning these two colours is an edge cut-set of $H'$. Let $H_1$ and $H_2$ be the two sides of the cut. Since $H'$ has at least one edge we can assume that it has an edge whose endpoints have colours $a$ and $b$. In particular, $H_1$ and $H_2$ have fewer edges than $H'$ and so, by induction, $(H_1, \sigma\vert_{V(H_1)})$ and $(H_2, \sigma\vert_{V(H_2)})$ are $F$-abundant. Applying \cref{lem:join_abundant} with $a' = b$ shows that $(H', \sigma\vert_{V(H')})$ is $F$-abundant.
\end{proof}

\subsection{\texorpdfstring{$F$}{F}-splittable graphs with chromatic number \texorpdfstring{$\chi(F)$}{chi(F)}}

In this subsection we prove \cref{thm:FabundantchiF}. We first define an operation that converts an $F$-abundant graph into a more complicated one. Let $\sigma$ be a surjective $F$-colouring of a graph $H$. For $e = ab \in E(F)$ we define $(H, \sigma)_e$ as the disjoint union of two copies of $(H, \sigma)$, say $(H_1, \sigma_1)$ and $(H_2, \sigma_2)$, where all vertices in $\sigma_1^{-1}(a)$ are joined to all vertices in $\sigma_2^{-1}(b)$ and all vertices in $\sigma_1^{-1}(b)$ are joined to all vertices in $\sigma_2^{-1}(a)$. 
If $(H, \sigma)$ is $F$-splittable, then so is $(H, \sigma)_e$.

Using this operation, we now construct a sequence of increasingly complicated $F$-splittable graphs. 
Let $e_1, \dotsc e_{e(F)}$ be an enumeration of the edges of $F$. We start with an $F$-splittable $F$-coloured graph $(H, \sigma)$ where $\sigma$ is surjective (for example, one could take $H$ to be $\abs{F}$ isolated vertices and $\sigma$ a bijection from $V(H)$ to $V(F)$). We define a sequence of graphs $(H^1, \sigma_1) = (H, \sigma)$, $(H^2, \sigma_2)$, \ldots, where $(H^{j + 1}, \sigma_{j + 1}) \coloneqq (H^j, \sigma_j)_{e_{j \pmod{e(F)}}}$ for each $j$. By \cref{thm:splittable_abundant}, the graph $H^m$ is $F$-abundant for every positive integer $m$.

\begin{lemma}\label{lem:Hm_homomorphisms}
    Let $G$ be a graph with $F \nrightarrow G$. Then, for all sufficiently large $m$, $H^m \nrightarrow G$. 
\end{lemma}

\begin{proof}
    Take $m$ sufficiently large in terms of $F$ and $G$. Suppose for a contradiction that there is a homomorphism $\varphi$ from $H^m$ to $G$. We view $\varphi$ as a $G$-colouring of $H^m$. For all $v \in V(F)$, let $S^m(v)$ be the set of colours appearing on the vertices corresponding to $v$ in $H^m$. Note that each colour is a vertex of $G$. We denote the two canonical copies of $H^i$ within $H^{i + 1}$ by $H^i_1$ and $H^i_2$.

    Now look inside $H^m = H^{m - 1}_1 \cup H^{m - 1}_2$. For $v \in V(F)$ and $i \in \set{1, 2}$, let $S^{m - 1}_i(v)$ be the set of colours appearing on vertices corresponding to $v$ in $H^{m - 1}_i$. Note that $S^m(v) = S^{m - 1}_1(v) \cup S^{m - 1}_2(v)$. Compare $\sum_{v \in V(F)} \abs{S^{m - 1}_1(v)}$ and $\sum_{v \in V(F)} \abs{S^{m - 1}_2(v)}$ and let $i \in \set{1, 2}$ be such that $\sum_{v \in V(F)} \abs{S^{m - 1}_i(v)}$ is the smaller (or equal) of the two sums. Define $H^{m - 1} \coloneqq H^{m - 1}_i$ and $S^{m - 1}(v) \coloneqq S^{m - 1}_i(v)$ for each $v \in V(F)$. Now look inside $H^{m - 1} = H^{m - 2}_1 \cup H^{m - 2}_2$ and iterate.

    Doing this produces a sequence of graphs $H^m \supset H^{m - 1} \supset \dotsb \supset H_1$ and, for each $v \in V(F)$, a sequence of sets of colours $S^m(v) \supset S^{m - 1}(v) \supset \dotsb \supset S^1(v)$ where $\sum_{v \in V(F)} \abs{S^j(v)}$ is the minimum of $\sum_{v \in V(F)} \abs{S^j_1(v)}$ and $\sum_{v \in V(F)} \abs{S^j_2(v)}$.

    We apply the pigeonhole principle to sequence of $k$-tuples $(S^j(v) : v \in V(F))$ where $j$ ranges from $1$ to $m$. Provided $m$ is sufficiently large, there will be more than $e(F)$ indices $j$ and some tuple $(S(v) : v \in V(F))$ such that $S_j(v) = S(v)$ for all $v \in V(F)$ and for all such $j$. Note that, crucially, $m$ just depends on $F$ and $G$. Since for each $v$ the sequence $S^j(v)$ is nested, we in fact will have an interval $I \subset \set{1, 2, \dotsc, m}$ of length greater than $e(F)$ such that $S_j(v) = S(v)$ for all $v \in V(F)$ and $j \in I$.

    Now, $F$ is not homomorphic to $G$ and each $S(v)$ is a non-empty set of colours (non-empty set of vertices of $G$). So, there must be adjacent vertices $x, y \in V(F)$ and non-adjacent vertices $c, d \in V(G)$ such that $c \in S(x)$ and $d \in S(y)$.
    Since $I$ is an interval of length greater than $e(F)$, there is a $j$ such that $j$ and $j + 1$ are both in $I$ and $e_{j \pmod{e(F)}} = xy$. We assume without loss of generality that $H^j = H^j_1$, so $S^j_1(v) = S^j(v) = S(v)$ for all $v \in V(F)$, and
    \begin{equation}\label{eq:Sineq}
        \sum_{v \in V(F)} \abs{S^j_1(v)} \leq \sum_{v \in V(F)} \abs{S^j_2(v)}.
    \end{equation}
    Now, for all $v \in V(F)$, $S_2^j(v) \subset S^{j + 1}(v) = S(v) = S_1^j(v)$. Thus, by \eqref{eq:Sineq}, we must have $S_2^j(v) = S_1^j(v) = S(v)$ for all $v \in V(F)$. Since $c \in S(x) = S_1^j(x)$, there is a vertex $u_x \in V(H_1^j)$ corresponding to $x$ with $\varphi(u_x) = c$. Similarly, since $d \in S(y) = S_2^j(y)$, there is a vertex $u_y \in V(H_2^j)$ corresponding to $y$ with $\varphi(u_y) = d$.

    Since $e_{j \pmod{e(F)}} = xy$, every vertex in $H_1^j$ corresponding to $x$ is adjacent to every vertex in $H_2^j$ corresponding to $y$ and so $u_x$ is adjacent to $u_y$. Since $\varphi$ is a homomorphism to $G$, $\varphi(u_x) = c$ must be adjacent to $\varphi(u_y) = d$. However, $c$ and $d$ were chosen to not be adjacent, which is the required contradiction.
\end{proof}

Since the graph $H^m$ is $F$-splittable by construction, this proves \cref{thm:FabundantchiF} and shows that $F$-abundance cannot be characterised by looking at right homomorphisms. We combine this with the following result of Ne\v{s}et\v{r}il and Zhu.

\begin{theorem}[{\cite[Thm.~1.1]{NZ04}}]\label{thm:homomorphisms_girth}
    For every graph $H'$ and for all positive integers $k$ and $\ell$ there exists a graph $H$ with the following properties:
    \begin{itemize}
        \item $H$ has girth greater than $\ell$, and
        \item for every graph $G$ on at most $k$ vertices, there is a homomorphism $H \to G$ if and only if there is a homomorphism $H' \to G$.
    \end{itemize}
\end{theorem}

\begin{corollary}\label{cor:abundant_girth}
    Let $\ell$ be a positive integer and $F$ and $G$ be graphs with $F \nrightarrow G$. There is an $F$-abundant graph $H$ with girth greater than $\ell$ and $H \nrightarrow G$. 
\end{corollary}

\begin{proof}
    By \cref{thm:FabundantchiF}, there is an $F$-abundant graph $H'$ with $H' \nrightarrow G$. Applying \cref{thm:homomorphisms_girth} with $k = \max\set{\abs{H'}, \abs{G}}$ gives a graph $H$ with girth greater than $\ell$, $H \to H'$, and $H \nrightarrow G$. Since $H'$ is $F$-abundant and $H \to H'$, the graph $H$ is also $F$-abundant.
\end{proof}

\subsection{Effective Abundance}\label{sec:effective}

Note that in the above arguments we were interested in showing abundance without paying too much attention to the degree of the polynomial in $\eps$. 
In some cases, we can actually show much better bounds on the degrees of these polynomials. This might be useful in the setting of very sparse graphs.
Below, we show how to do this in the most basic case of the triangle-abundance of $C_5$. First, we show that we can find many paths with four vertices of a certain type in a graph that is far from triangle-free.

\begin{proposition}\label{Prop: paths3}
    Let $p\geq 100/n$ and let $G$ be an $n$-vertex tripartite graph with tripartition $V(G)=A\cup B\cup C$ that contains at least $pn^2$ edges-disjoint triangles. Then, $G$ contains at least $\Omega(p^3n^4)$ paths of length $3$ of the form $x_1x_2x_3x_4$ with $x_1,x_4\in A$, $x_2\in B$, and $x_3\in C$.
\end{proposition}

\begin{proof}
    We may assume that $E(G)$ is the disjoint union of $pn^2$ edge-disjoint triangles. Observe that for every $x\in A$, $d_B(x)=d_C(x)$ (and analogously for $B$ and $C$). 
    
    Simple counting gives that the number of $P_3$'s with middle edge in $G[B,C]$ is at least
    \begin{equation*}
        \sum_{(x,y)\in E(B,C)} d_A(x) (d_A(y)-1) \ge \sum_{(x,y)\in E(B,C)} (d_C(x)-1) (d_B(y)-1) = \# P_3\text{'s in } G[B,C].
    \end{equation*}
    Using the fact that $P_3$ is Sidorenko (see e.g.\ \cite[Theorem~5.5.10]{Zhaobook}) and $E(G[B,C])=pn^2$, there must be at least $\Omega(p^3n^4)$ paths of length $3$ as required. 
\end{proof}

We now use a strategy similar to \cref{lem:peel_abundant} for extending this result to show that a graph that is far from triangle-free contains many copies of $C_5$.

\begin{lemma}\label{lem:effective}
    Let $100n^{-1/2}\leq p \le n^{-\gamma}$ for some $\gamma > 0$. 
    Then, for every $\e'>0$, there exists $C(\e') > 0$ such that every $n$-vertex graph $G$ which is $p$-far from being triangle-free contains $C(\e') p^{6+\e'} n^5$ copies of $C_5$. 
\end{lemma}

\begin{proof}
    Let $\e = \e' \gamma$ and choose $\delta > 0$ sufficiently small. We randomly partition the vertices of $G$ into three parts of equal size. The probability that a triangle of $G$ has one vertex in each part of the partition is $2/9 > 1/5$, and so there exists a partition $V(G) = A \cup B \cup C$ such that $G$ has $(p/5) n^2$ edge-disjoint triangles with one vertex in each part. We delete all other edges from the graph. For a graph $H$, denote by $m(H)$ the number of edge-disjoint triangles in $H$.

    We now perform the following iterative process. We start with $A_0 = A$ and $p_0 = m(G) / (n \abs{A})$. Then, given $A_i \subset A$ and $p_i$, let $A_{i+1} \subset A_i$ be the set of those vertices that are contained in at least $\delta p_i n$ edge-disjoint triangles, and let $p_{i+1} = m(G[A_{i+1} \cup B \cup C]) / (n \abs{A_{i+1}})$. If $\abs{A_{i+1}} \ge \delta \abs{A_i}$, we stop at this step, and otherwise we continue.
    
    Define $G_i = G[A_i \cup B \cup C]$, so $p_i = m(G_i) / (n \abs{A_i})$. Note that $\abs{A_i} \le \delta^i n$. Moreover,
    \begin{equation*}
        m(G_{i+1}) \ge m(G_i) - \delta p_i n \abs{A_i} = (1 - \delta) m(G_i)
    \end{equation*}
    and so $m(G_i) \ge (1 - \delta)^i m(G) \ge (1 - \delta)^i (p/5) n^2$.

    \begin{claim}
        The process stops after $t \le \log(n) / \log(1/\delta)$ steps.
    \end{claim}

    \begin{proof}
        Indeed, after $i > \log(n) / \log(1/\delta)$ steps, we have $\abs{A_i} = \delta^i n < 1$ and so $\abs{A_i} = 0$. However, $m(G_i) \ge (1 - \delta)^i (p/5) n^2 > 0$ which gives a contradiction.
    \end{proof}

    Note that if $\delta$ is sufficiently small, then $m(G_t) \ge (1 - \delta)^t (p/5) n^2 \ge (p/5) n^{2-\e/9}$. In particular, $p_t \ge (p/5) n^{1-\e/9} / \abs{A_t}$, and so every vertex in $A_{t+1}$ is contained in at least $(\delta p / 5) n^{2-\e/9} / \abs{A_t}$ edge-disjoint triangles. Now, we apply the exact same argument to $G_{t+1}$ but with $B$ instead of $A$. This yields a graph $G' = G[A_{t+1} \cup B_{s+1} \cup C]$ for some subset $B_{s+1} \subset B$ with $m(G') \ge (p/5) n^{2-2\e/9}$ such that every vertex in $B_{s+1}$ is contained in at least $(\delta p / 5) n^{2-2\e/9} / \abs{B_s}$ edge-disjoint triangles and $\abs{B_{s+1}} \ge \delta \abs{B_s}$.

    Now, for each $x \in B_{s+1}$, consider the graph $G'' = G[(N_G(x) \cap A_{t+1}) \cup B \cup C]$. Note that $\abs{N_G(x) \cap A_{t+1}} \ge (\delta p / 5) n^{2-2\e/9} / \abs{B_s}$ since $x$ is contained in at least that many edge-disjoint triangles. So, by construction, we get $m(G'') \ge (\delta p / 5) n^{2-2\e/9} / \abs{B_s} \cdot (\delta p / 5) n^{2-\e/9} / \abs{A_t} = q n^2$ where $q = (\delta p / 5)^2 n^{2-\e/3} / (\abs{A_t} \abs{B_s})$. Applying \cref{Prop: paths3} shows that there are at least $\Omega(q^3 n^4)$ paths in $G''$, each of which can be extended by $x$ to a copy of $C_5$. Since $x \in B_{s+1}$ was arbitrary, it follows that in total there are at least
    \begin{equation*}
        \Omega(q^3 n^4 \abs{B_{s+1}}) = \Omega\left(\frac{\delta^6 p^6 n^{10-\e} \abs{B_{s+1}}}{\abs{A_t}^3 \abs{B_s}^3}\right) \ge \Omega(\delta^7 p^6 n^{5-\e}) \ge \Omega(\delta^6 p^{6+\e'} n^5) = C(\e') p^{6+\e'} n^5
    \end{equation*}
    copies of $C_5$ in $G$.
\end{proof}

We remark that in the special case that $p$ is of the order $n^{-1/2}$ the best results known gives on the order of $n^4$ copies of $C_5$, which is tighter by a factor of $n^{\eps'/2}$. 
The tightest result is due to Conlon, Fox, Sudakov and Zhao \cite{conlon2021regularity}, and makes use of sparse regularity.
They also give a construction of a graph containing $o(n^{2.442})$ copies of $C_5$ that cannot be made triangle free by deleting $o(n^{3/2})$ edges. 
This shows in particular that one cannot replace $p^{6+\eps'}$ by $p^5$ in \cref{lem:effective}.
It would be interesting to have similar constructions that work for larger $p$.

\section{Additive results}
\label{sec:additive}

In this section, we turn to a problem that is seemingly unrelated to abundance. We ask how large sets of integers can be without having solutions to certain equations, proving \cref{thm:genus2sqrt}. To be precise, consider an equation $E$ of the form
\begin{equation*}
    a_1 x_1 + \dots + a_k x_k = 0
\end{equation*}
with $a_i \neq 0$ for $i \in [k]$ but $a_1 + \dots + a_k = 0$. Ruzsa~\cite{Ruzsa} studied the following two quantities:
\begin{itemize}
    \item the size $R_E(N)$ of a largest set $\cA \subset [N]$ such that equation $E$ has no solutions with distinct integers $x_1, \dots, x_k \in \cA$, and
    \item the size $r_E(N)$ of a largest set $\cA \subset [N]$ with only trivial solutions to equation $E$, where a solution $x_1, \dots, x_k \in \cA$ is called \defn{trivial} if all maximal subsets $T \subset [k]$ with $x_i = x_j$ for all $i, j \in T$ satisfy $\sum_{i \in T} a_i = 0$.
\end{itemize}
Clearly, $r_E(N) \le R_E(N)$. In \cref{sec:eqtographs}, we explain why these quantities are important for abundance, but in short we want to know which equations satisfy $R_E(N) = N^{1-o(1)}$. In this context, the genus of an equation is important.

\begin{definition}
    The \defn{genus} of equation $E$ is the largest integer $m$ such that there is a partition $\cT_1 \cup \dotsb \cup \cT_m$ of the set of indices $[k]$ into $m$ disjoint nonempty sets $\cT_j$ such that, for every $j$,
    \begin{equation*}
        \sum_{i \in \cT_j} a_i = 0.
    \end{equation*}
\end{definition}

Ruzsa proved the following.

\begin{theorem}
     For any equation $E$ of genus $m$ we have
     \begin{equation*}
         r_E(N) = \cO(N^{1/m}).
     \end{equation*}
\end{theorem}

Unfortunately, he did not prove the same result for $R_E(N)$. If the equation is \defn{symmetric}, which means that it can be written as
\begin{equation*}
    \sum_{i \in [\ell]} a_i x_i = \sum_{i \in [\ell]} a_i x_{\ell+i},
\end{equation*}
then $r_E(N) = \cO(N^{1/\ell})$, since symmetric equations in $2 \ell$ variables have genus $\ell$. Moreover, Ruzsa showed the following.

\begin{theorem}
    Let $\ell \geq 2$. For a symmetric equation $E$ in $2 \ell$ variables we have
    \begin{equation*}
        R_E(N) = \cO(\sqrt{N}).
    \end{equation*}
\end{theorem}

Ruzsa also showed that this bound on $R_E(N)$ is the best possible \emph{general} bound for symmetric equations. Bukh~\cite{Bukh} later improved this bound for individual symmetric equations.

Based on these results, Ruzsa~\cite[\S~9]{Ruzsa} asked whether $r_E(N) \ge N^{1/m-o(1)}$ holds for all equations of genus $m$. In that case, equations of genus one would satisfy $R_E(N) \ge r_E(N) \ge N^{1-o(1)}$. However, so far, the only equations known to have $R_E(N) = N^{1-o(1)}$ are \defn{convex} equations, which are equations with exactly one negative coefficient. The corresponding set $\cA$ can then be obtained by adapting Behrend's construction~\cite{Behrend} of a large set with only trivial arithmetic progressions.

Note that according to the above results, it was still possible that some nonsymmetric equation of genus at least two satisfied $R_E(N) \ge N^{1-o(1)}$. We make progress on this by proving that in fact $R_E(N) = \cO(\sqrt{N})$ holds for all equations of genus at least two; this is \cref{thm:genus2sqrt}.

We first sketch the idea of this proof for the equation
\begin{equation*}
    x_1 - x_2 + y_1 + y_2 - 2 y_3 = 0.
\end{equation*}
Assume that $\cA \subset [N]$ is a subset of size $\abs{\cA} \ge C \sqrt{N}$. To find a solution to the above equation with distinct integers from $\cA$, we will construct an auxilliary graph. This will be a 4-partite graph whose parts are $\cA$, $\cA + \cA$, $\cA + 2 \cA$, and $2 \cA$, and for all $x, y \in \cA$ we will add the edges of the following path to our graph:
\begin{equation*}
    x, \quad x + y, \quad x + 2 y, \quad 2 y.
\end{equation*}
In total, this adds $C^2 N$ many paths and edges to the graph. So, if $C$ is sufficiently large, the average degree of this graph will be larger than any constant that we want, and we may pass to a subgraph where the minimum degree is larger than any constant that we want.

Now, by construction, any path across the four parts of this subgraph will be of the form
\begin{equation*}
    x_1, \quad x_1 + y_1, \quad x_1 + y_1 + y_2, \quad x_1 - x_2 + y_1 + y_2
\end{equation*}
where $x_1, x_2, y_1, y_2 \in \cA$. Moreover, since the vertex $x_1 - x_2 + y_1 + y_2$ is in the last part of the graph, we must have $x_1 - x_2 + y_1 + y_2 \in 2 \cA$ and so $x_1 - x_2 + y_1 + y_2 = 2 y_3$ for some $y_3 \in \cA$. But this means that $x_1 - x_2 + y_1 + y_2 - 2 y_3 = 0$, which implies that such a path corresponds to a solution of the equation with integers from $\cA$. It only remains to ensure that all integers are distinct. But this is not difficult: if we choose the path in the graph incrementally, every outgoing edge of a vertex corresponds to different value for one of the variables. Since the degree of every vertex is large, this means we can simply choose an edge where the value of the new variable is different from the values of all previous variables, which then yields a solution to the equation with distinct integers from $\cA$.
This concludes the sketch. 
In the general case, the proof is simply a generalisation of this argument with a few more details added.

\begin{remark}
    This argument is essentially a translation of the proof of \cref{lem:glue_blowup_abundant} to the additive setting. 
    One could view it as replacing an edge of a $C_4$ by path on $3$ vertices.
    The essential difference is that the additive structure allows one to find the optimal number of edges to be replaced by disjoint paths.
\end{remark}

\begin{theorem}
    Let $E$ be an equation of genus at least two. Then, $R_E(N) = \cO(\sqrt{N})$.
\end{theorem}

\begin{proof}
    Since $E$ is an equation of genus at least two, we may write $E$ as
    \begin{equation*}
        \sum_{i=1}^s a_i x_i + \sum_{j=1}^t b_j y_j = 0
    \end{equation*}
    where $\sum_{i=1}^s a_i = \sum_{j=1}^t b_j = 0$. Let $\cA \subset [N]$ and write $M = \abs{\cA}$. We want to show that if $M > C \sqrt{N}$ for some sufficiently large constant $C$ depending only on $E$, then there exists a solution to $E$ with distinct integers from $\cA$.

    For this, we define an auxiliary $(s+t-1)$-partite graph $G$ as follows. Its vertex sets $V_1, \dots, V_{s+t-1}$ are defined as
    \begin{equation*}
        V_k = a_1 \cA + \sum_{j=1}^{k-1} b_j \cA \text{ for } k \in [t-1] \quad \text{ and } \quad V_{k+t-1} = \sum_{i=k+1}^s (-a_i \cA) - b_t \cA \text{ for } k \in [s].
    \end{equation*}
    For every ordered pair $(x, y)$ of distinct $x, y \in \cA$, we add a $V_1 V_2 \dots V_{s+t-1}$ path to $G$ whose vertices $v_1 v_2 \dots v_{s+t-1}$ are given by
    \begin{equation*}
        v_k = a_1 x + \sum_{j=1}^{k-1} b_j y \text{ for } k \in [t-1] \quad \text{ and } \quad v_{k+t-1} = \sum_{i=k+1}^s (-a_i x) - b_t y \text{ for } k \in [s],
    \end{equation*}
    as shown in \cref{fig:path}.
    
    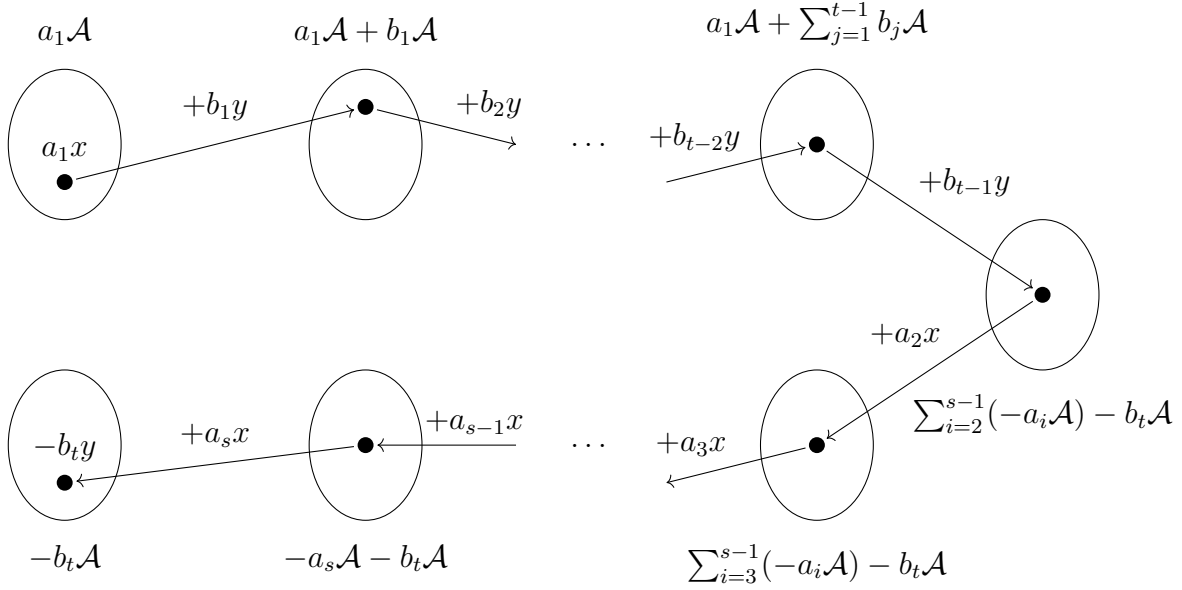
\begin{figure}[ht!]
        \centering
        \begin{tikzpicture}[set/.style={draw,ellipse,minimum height=2cm,minimum width=1.5cm,outer sep=0.2cm},every node/.style={inner sep=0.075cm,outer sep=0.05cm}]
            \node[set,label={above:$a_1 \cA$}] at (0,0) {};
            \node[set,label={above:$a_1 \cA + b_1 \cA$}] at (4,0) {};
            \node[set,label={above:$a_1 \cA + \sum_{j=1}^{t-1} b_j \cA$}] at (10,0) {};
            \node[set,label={below:$\sum_{i=2}^{s-1} (-a_i \cA) - b_t \cA$}] at (13,-2) {};
            \node[set,label={below:$\sum_{i=3}^{s-1} (-a_i \cA) - b_t \cA$}] at (10,-4) {};
            \node[set,label={below:$-a_s \cA - b_t \cA$}] at (4,-4) {};
            \node[set,label={below:$-b_t \cA$}] at (0,-4) {};
            \node at (7,0) {$\dots$};
            \node at (7,-4) {$\dots$};
            \node[fill,circle,label=above:$a_1 x$] (S1) at (0,-0.5) {};
            \node[fill,circle] (S2) at (4,0.5) {};
            \node[fill,circle] (S3) at (10,0) {};
            \node[fill,circle] (S4) at (13,-2) {};
            \node[fill,circle] (S5) at (10,-4) {};
            \node[fill,circle] (S6) at (4,-4) {};
            \node[fill,circle,label=above:$-b_t y$] (S7) at (0,-4.5) {};
            \draw[->] (S1) edge["$+ b_1 y$"above=0.2] (S2) (S2) edge["$+ b_2 y$"] (6,0) (8,-0.5) edge["$+ b_{t-2} y$"pos=0.6] (S3) (S3) edge["$+ b_{t-1} y$"pos=0.4] (S4) (S4) edge["$+ a_2 x$"'pos=0.4] (S5) (S5) edge["$+ a_3 x$"'] (8,-4.5) (6,-4) to["$+ a_{s-1} x$"'pos=0.3] (S6) (S6) edge["$+ a_s x$"above=0.1] (S7);
        \end{tikzpicture}
        \caption{The path for the ordered pair $(x, y)$ of distinct $x, y \in \cA$.} \label{fig:path}
    \end{figure}

    We observe that $v_1 = a_1 x$, $v_{k+1} = v_k + b_k y$ for $k \in [t-1]$, and $v_{k+t} = v_{k+t-1} + a_{k+1} x$ for $k \in [s-1]$. This is obvious except for $v_t = v_{t-1} + b_{t-1} y$, but there we have
    \begin{equation*}
        v_t = \sum_{i=2}^s (-a_i x) - b_t y = a_1 x + \sum_{j=2}^{t-1} b_j y = v_{t-1} + b_{t-1} y,
    \end{equation*}
    where we used $\sum_{i=1}^s a_i = \sum_{j=1}^t b_j = 0$. In particular, the difference between the two endpoints of any edge in this path uniquely determines either $x$ or $y$. Moreover, once we have one of these two values, we can determine the other value by looking at either endpoint of the edge. So, from any edge we can determine both $x$ and $y$, which implies that all of these paths are edge-disjoint. We denote the collection of these paths by $\cP$, so we have $\abs{\cP} = M (M-1) \ge M^2 / 2$.

    We claim that any $V_1 V_2 \dots V_{s+t-1}$ path in $G$ yields a solution to the equation $E$. Indeed, if $u_1 u_2 \dots u_{s+t-1}$ is such a path, then by the observation from above there are $x_i, y_j \in \cA$ with $u_1 = a_1 x_1$, $u_{k+1} = u_k + b_k y_k$ for $k \in [t-1]$, and $u_{k+t} = u_{k+t-1} + a_{k+1} x_{k+1}$ for $k \in [s-1]$. Moreover, $u_{s+t-1} = -b_t y_t$ for some $y_t \in \cA$ by definition of $V_{s+t-1}$. Therefore,
    \begin{equation*}
        \sum_{i=1}^s a_i x_i + \sum_{j=1}^t b_j y_j = a_1 x_1 + \sum_{j=1}^{t-1} b_j y_j + \sum_{i=2}^s a_i x_i + b_t y_t = u_{s+t-1} + b_t y_t = 0,
    \end{equation*}
    as claimed. It remains to show that there exists such a path where the corresponding variables are distinct.

    The graph $G$ has at most $n = c N$ many vertices where $c = (s+t-1) (\sum_{i=1}^s \abs{a_i} + \sum_{j=1}^t \abs{b_j})$. We pass to a subgraph $H$ of $G$ as follows. While there exists a vertex $v$ of $G$ that is contained in less than $\abs{\cP} / 2 n$ many paths of $\cP$, we delete all paths containing $v$. Since this deletes at most $n \cdot (\abs{\cP} / 2 n) = \abs{\cP} / 2$ many paths, we are left with at least $\abs{\cP} / 2$ many paths $\cP' \subset \cP$ and every vertex is contained in either none or at least $\abs{\cP} / 2 n$ many of these paths. Let $H$ be the (non-empty) subgraph of $G$ induced by the vertices in $\cP'$. In particular, every vertex $v \in V(H) \cap V_k$ for $k \in [s+t-2]$ has at least $\abs{\cP} / 2 n \ge M^2 / 4 c N \ge C^2 / 4 c$ many neighbours in $V(H) \cap V_{k+1}$. If $C$ is sufficiently large, this is at least $2 (s+t)$.

    We now construct the path $u_1 u_2 \dots u_{s+t-1}$. We start with any vertex $u_1 = a_1 x_1 \in V(H) \cap V_1$. Then, for $k \in [t-1]$, we pick a neighbour $u_{k+1} = u_k + b_k y_k \in V(H) \cap V_{k+1}$ of $u_k$ such that $y_k$ is distinct from $x_1, y_1, \dots, y_{k-1}$. Since $u_k$ has more than $k$ neighbours in $V(H) \cap V_{k+1}$ and every neighbour corresponds to a different value of $y_k$, this is always possible. Afterwards, for $k \in [s-2]$, we pick a neighbour $u_{k+t} = u_{k+t-1} + a_{k+1} x_{k+1} \in V(H) \cap V_{k+t}$ of $u_{k+t-1}$ such that $x_{k+1}$ is distinct from $x_1, \dots, x_k, y_1, \dots, y_{t-1}$. Again, $u_{k+t-1}$ has more than $k+t-1$ neighbours in $V(H) \cap V_{k+t}$, so this is possible. Finally, we pick a neighbour $u_{s+t-1} = u_{s+t-2} + a_s x_s = -b_t y_t \in V(H) \cap V_{s+t-1}$ of $u_{s+t-2}$ such that $x_s$ and $y_t$ are both distinct from $x_1, \dots, x_{s-1}, y_1, \dots, y_{t-1}$. As $u_{s+t-2}$ has more than $2 (s+t-2)$ many neighbours in $V(H) \cap V_{s+t-1}$, this is possible. Also, $x_s$ and $y_t$ are distinct: this is true for all edges between $V_{s+t-2}$ and $V_{s+t-1}$ in $G$ since all paths in $\cP$ were constructed with $x$ and $y$ distinct. Thus, $x_1, \dots, x_s, y_1, \dots, y_t$ are all distinct, and we know from above that
    \begin{equation*}
        \sum_{i=1}^s a_i x_i + \sum_{j=1}^t b_j y_j = 0. \qedhere
    \end{equation*}
\end{proof}

In fact, the above proof immediately gives something stronger.
We say an equation $E$ in $\ell$ variables is \defn{abundant} if there exists constants  $\gamma,C>0$, depending only on $E$, such that for all $\eps \in (0,1)$ and $N\in \mathbb{N}$, any subset $\cA\subset [N]$ of order at least $\eps N$ contains at least $\gamma\eps^C N^{\ell-1}$ solutions to $E$. 
Note that abundance implies unavoidability. 
The reverse implication is an interesting open question (see \cref{sec:questions}).

\begin{corollary}\label{cor:abundance}
    All equations of genus at least $2$ are abundant. 
\end{corollary}

\begin{proof}
    Suppose $\cA \subset [N]$ and $|\cA| = M = \eps N$, and let $E$ be an equation of genus at least $2$.
    Then we proceed as in the above proof to find an $(s+t-1) = (\ell - 1 )$-partite subgraph $(V_1,\dots,V_{\ell-1})$, with the property that each $v \in V_i$ has at least $M^2/4cN > \eps^2N/4c$ neighbours in $V_{i-1}$ and $V_{i+1}$. 
    Thus, $V_1$ has at least $\eps^2N/4c$ vertices and choosing the first vertex of the path arbitrarily and then extending arbitrarily yields $(\eps^2N/4c)^{\ell-1}$ paths via $V_1V_2\dots V_{\ell-1}$.
    As $c$ only depends on $E$, this completes the proof. 
\end{proof}

\section{From graphs to equations and back}
\label{sec:eqtographs}

Thus far we discussed abundance and the problem of integer solutions to equations largely in isolation. In reality, these two subjects are intimately connected. This has been known since at least the work of Ruzsa and Szemer\'{e}di~\cite{RS78} who used Behrend's construction~\cite{Behrend} of a large set with only trivial arithmetic progressions to show that triangles are not triangle-abundant. At the moment, the Ruzsa-Szemer\'{e}di construction is the only technique we have for showing that a graph $H$ with $H \to F$ is not $F$-abundant. So far, however, it remained possible that this is simply a convenient way of proving non-abundance and nothing more. In this section, we explore this connection in more depth and show that in some cases the avoidability of equations associated to $H$ can imply abundance and even characterise it.
We also show that a strategy for proving non-abundance suggested by Gishboliner, Shapira and Wigderson \cite{GSW23} cannot succeed, and that their random construction of a strongly genus one graph can be replaced by a small explicit graph. 

We begin by explaining the Ruzsa-Szemer\'{e}di construction in full generality. Fix an injection $c \colon V(F) \to [n]$. Then, for any integer $N$ and any subset $\cA \subset [N]$, the \defn{Ruzsa-Szemer\'edi graph} $\RS(F, N, \cA)$ is the $\abs{F}$-partite graph with vertex sets $V_v = [n N]$ for $v \in V(F)$ whose edges between $V_u$ and $V_v$ for $u v \in E(F)$ are those pairs $x y \in V_u \times V_v$ that satisfy $y - x \in (c(v) - c(u)) \cA$.

It is easy to see that this graph contains many edge-disjoint copies of $F$.
Indeed, let $x \in [N]$ and $a \in \cA$ be arbitrary.
Then, for all $v\in V(F)$, let $x_v = x + (c(v)-1) a \in V_v$.
This set of vertices induces a copy of $F$, and every copy so generated is edge-disjoint. 
Therefore, this graph is far from $F$-free if $\abs{\cA}$ is sufficiently large. If we want to show that $H$ is not $F$-abundant, we need to ensure that this graph also contains few (not necessarily edge-disjoint) copies of $H$. 
For that, consider an $F$-colouring $\sigma$ of $H$ and associate to any cycle $v_1 \dots v_\ell$ of $H$ its \defn{cycle-equation}
\begin{equation*}
    (c(\sigma(v_2)) - c(\sigma(v_1))) x_{v_1 v_2} + \dots + (c(\sigma(v_\ell)) - c(\sigma(v_{\ell-1}))) x_{v_{\ell-1} v_\ell} + (c(\sigma(v_1)) - c(\sigma(v_\ell))) x_{v_\ell v_1} = 0.
\end{equation*}
We denote this system of equations by $\Eq(H, \sigma, c)$. In every embedding of $(H,\sigma)$ in the Ruzsa-Szemer\'edi graph, the embedding of a cycle corresponds to a solution of its cycle-equation with integers from $\cA$.
So, if $\cA$ has few solutions to $\Eq(H, \sigma, c)$, the graph contains few copies of $(H,\sigma)$.

With some additional work, it can  be shown that the Ruzsa-Szemer\'{e}di construction implies: 
\begin{quote}
      If $(H,\sigma)$ is $F$-abundant, then $\cE(H,\sigma,c)$ is abundant.
\end{quote}
Conversely, if $\cE(H,\sigma,c)$ is not abundant, this shows that $(H,\sigma)$ is not $F$-abundant. Thus, \cref{cor:abundance} is directly relevant to our investigations of graph abundance. It implies that a single equation of genus at least two does not suffice for the Ruzsa-Szemer\'edi construction.

\subsection{From graphs to equations}

Ruzsa's question, \cref{qu:Ruzsa}, asks whether all equations of genus one satisfy $R_E(N) = N^{1-o(1)}$. Under the assumption that the answer is yes, Gishboliner, Shapira, and Wigderson~\cite{GSW23} showed that every $K_3$-colouring of an appropriate triangle-free random graph $G$ has an associated system of equations which can be linearly combined to obtain an equation with genus one, implying that $G$ is not $K_3$-abundant (see \cite[Section 3.2]{GSW23}).

We explicitly construct a graph with this property. Let $G_n$ be the $3$-partite graph with vertex sets $A = [n]$, $B = 2^{[n]}$, and $C = 2^{[n]}$. Then, join $a \in A$ to exactly those $X \in B$ and $Y \in C$ that satisfy $a \in X$ and $a \in Y$. Moreover, join $X \in B$ to $Y \in C$ if and only if $X \cap Y = \emptyset$. It is easy to see that this graph is triangle-free. Moreover, it can be shown that each tripartite triangle-free graph is a subgraph of $G_n$ for $n$ sufficiently large. 
Thus, if any tripartite triangle-free graph is not $K_3$-abundant, then $G_n$ will not be $K_3$-abundant for some sufficiently large $n$.

In fact, a computer can check that the subgraph of $G_3$ depicted in \cref{fig:g3} has the property that every $K_3$-colouring has a linear combination with genus one. So, this subgraph (and therefore $G_3$) is not $K_3$-abundant if Ruzsa's question has a positive answer.

\begin{figure}[ht!]
    \centering
    \begin{tikzpicture}[every node/.style={circle,fill,inner sep=0.075cm,outer sep=0.05cm}]
        \node["$1$"above] (A1) at (-1,2) {};
        \node["$2$"above] (A2) at (0,2) {};
        \node["$3$"above] (A3) at (1,2) {};
        \node["$\set{1}$"below left] (B1) at (-4.5,-0.5) {};
        \node["$\set{2}$"below left] (B2) at (-3.9,-1.2) {};
        \node["$\set{3}$"below left] (B3) at (-3.3,-1.9) {};
        \node["$\set{1,2}$"below left] (B12) at (-2.7,-2.6) {};
        \node["$\set{1,3}$"below left] (B13) at (-2.1,-3.3) {};
        \node["$\set{2,3}$"below left] (B23) at (-1.5,-4) {};
        \node["$\set{1}$"below right] (C1) at (1.5,-4) {};
        \node["$\set{2}$"below right] (C2) at (2.1,-3.3) {};
        \node["$\set{3}$"below right] (C3) at (2.7,-2.6) {};
        \node["$\set{1,2}$"below right] (C12) at (3.3,-1.9) {};
        \node["$\set{1,3}$"below right] (C13) at (3.9,-1.2) {};
        \node["$\set{2,3}$"below right] (C23) at (4.5,-0.5) {};
        \draw (A1) edge (B1) edge (B12) edge (B13) (A2) edge (B2) edge (B12) edge (B23) (A3) edge (B3) edge (B13) edge (B23);
        \draw (A1) edge (C1) edge (C12) edge (C13) (A2) edge (C2) edge (C12) edge (C23) (A3) edge (C3) edge (C13) edge (C23);
        \draw (B1) edge (C2) edge (C3) edge (C23) (B2) edge (C1) edge (C3) edge (C13) (B3) edge (C1) edge (C2) edge (C12) (B12) edge (C3) (B13) edge (C2) (B23) edge (C1);
    \end{tikzpicture}
    \caption{A subgraph of $G_3$ obtained by removing the empty and complete sets.} \label{fig:g3}
\end{figure}
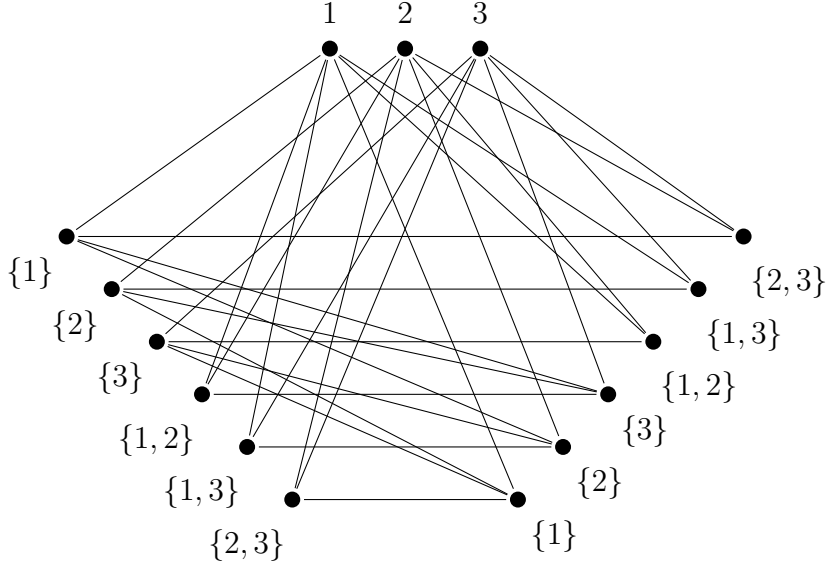

It would be interesting to prove that some triangle-free graph is not $K_3$-abundant without relying on the assumption that Ruzsa's question has a positive answer. Unfortunately, as mentioned in \cref{sec:additive}, the only equations known to have $R_E(N) = N^{1-o(1)}$ are convex equations, which are equations with exactly one negative coefficient. Gishboliner, Shapira, and Wigderson therefore suggested trying to find a triangle-free graph where every $K_3$-colouring has an associated system that can be linearly combined to give a convex equation. 
This would unconditionally prove that such a graph is not $K_3$-abundant. They tried to find such a graph with a computer search, but were unsuccessful.

We show that this cannot work. Indeed, we prove that if some linear combination of the equations from $\Eq(H, \sigma, c)$ is convex, then there is a single cycle that yields a convex equation. In $H$, the colours associated to such a cycle by $c \circ \sigma$ are in strictly increasing order, so the cycle has length at most $3$ and must therefore be a triangle. In general, this means that we can currently only show that a graph $H$ is not $F$-abundant if every $F$-colouring of $H$ has an increasing cycle, a property that is called \defn{increasing-cycle-unavoidable} (see \cite[Section 1.5]{GSW23}).

\begin{lemma}\label{lem:convex_not_possible}
    Let $(H, \sigma)$ be an $F$-coloured graph and $c \colon V(F) \to [n]$ be an injection. If some linear combination of the equations from $\Eq(H, \sigma, c)$ is convex, then there is a single cycle-equation of $H$ that is convex.
\end{lemma}

\begin{proof}
    Let
    \begin{equation*}
        \sum_{e \in E(H)} a_e x_e = 0
    \end{equation*}
    be a linear combination of the cycle-equations from $\Eq(H, \sigma, c)$ that is convex. Let $v \in V(H)$ be any vertex. Then, we claim that
    \begin{equation*}
        S(v) = \sum_{u \in N(v)} \frac{a_{v u}}{c(\sigma(u)) - c(\sigma(v))} = 0.
    \end{equation*}
    Indeed, this holds for any single cycle-equation: if $v$ appears in the cycle, the contribution to $S(v)$ of the two edges incident to $v$ in the cycle cancel each other out which means that $S(v)$ is zero as claimed. So this holds for a single cycle-equation, and therefore also for any linear combination of them.

    \begin{claim}\label{clm:posincidentedges}
        Let $v$ be such that $a_{v u} \ge 0$ for all $u \in N(v)$. Then, if some $u \in N(v)$ satisfies $c(\sigma(u)) < c(\sigma(v))$ and $a_{v u} > 0$, then some $u' \in N(v)$ satisfies $c(\sigma(u')) > c(\sigma(v))$ and $a_{v u'} > 0$.
    \end{claim}

    \begin{proof}
        The contribution of $u$ to $S(v)$ is negative. Since $S(v) = 0$, there must therefore exist some $u' \in N(v)$ such that the contribution of $u'$ to $S(v)$ is positive. Given that $a_{v u'} \ge 0$, it must hold that $a_{v u'} > 0$ and thus also $c(\sigma(u')) - c(\sigma(v)) > 0$. This proves the claim.
    \end{proof}

    Now, let $v u \in E(H)$ be such that $c(\sigma(v)) < c(\sigma(u)$) and $a_{v u} < 0$. Note that the contribution of $u$ to $S(v)$ is negative. Since $S(v) = 0$, there must therefore exist some $w \in N(v) \setminus \{u\}$ whose contribution to $S(v)$ is positive. Since the linear combination of the cycle-equations is convex, $a_{v u}$ is the only coefficient that is negative, and so we must have $a_{v w} > 0$. In order for the contribution of $w$ to $S(v)$ to be positive, it then follows that $c(\sigma(w)) > c(\sigma(v))$.
    
    Pick a maximal path $v_1 \dots v_\ell$ with $v_1 = v$, $c(\sigma(v_1)) < \dots < c(\sigma(v_\ell))$, and $a_{v_i v_{i+1}} > 0$ for all $i \in [\ell-1]$. The path $v w$ shows that $\ell \ge 2$. If $v_\ell \neq u$, then we know that $a_{v_\ell u'} \ge 0$ for all $u' \in N(v_\ell)$. Moreover, note that $v_{\ell-1} \in N(v_\ell)$ satisfies $c(\sigma(v_{\ell-1})) < c(\sigma(v_\ell))$ and $a_{v_{\ell-1} v_\ell} > 0$ by the choice of the path. Thus, \cref{clm:posincidentedges} implies that there exists some $v_{\ell+1} \in N(v_\ell)$ with $c(\sigma(v_{\ell+1})) > c(\sigma(v_\ell))$ and $a_{v_\ell v_{\ell+1}} > 0$. This contradicts the maximality of the path.

    Therefore, it must hold that $v_\ell = u$. But then $v_1 \dots v_\ell$ is a cycle with $c(\sigma(v_1)) < \dots < c(\sigma(v_\ell))$, which means that the associated cycle-equation is convex as required.
\end{proof}

\subsection{From equations to graphs}

Thus far we have only discussed cases where non-abundance of $\cE(H,\sigma,c)$ determines non-abundance of $(H,\sigma)$.
But in general, we may ask whether certain properties of the equations $\Eqs(H, \sigma) = \bigcup_c \set{\Eq(H, \sigma, c)}$ associated to an $F$-coloured graph $(H, \sigma)$ determine whether this graph is $F$-abundant or not. 
First, we consider the special case where the graph $H$ is a cycle $C$. In that case, each $\Eq(C, \sigma, c)$ is a single equation, and we show that the properties of $\Eqs(H, \sigma)$ determine $F$-abundance.

\begin{theorem}\label{thm:cycleabundance}
    Let $(C, \sigma)$ be an $F$-coloured cycle. The following are equivalent:
    \begin{enumerate}[ref=(\arabic*)]
        \item\label{thm:cycleabundance-itm1} $(C,\sigma)$ is $F$-abundant,
        \item\label{thm:cycleabundance-itm2} No equation from $\Eqs(C, \sigma)$ is convex, and
        \item\label{thm:cycleabundance-itm3} Every equation from $\Eqs(C, \sigma)$ has genus at least two.
    \end{enumerate}
\end{theorem}

\begin{proof}
    \ref{thm:cycleabundance-itm1} $\Rightarrow$ \ref{thm:cycleabundance-itm2}: If $\Eq(C, \sigma, c)$ is convex, we know from above that this would allow us to form a Ruzsa-Szemer\'{e}di graph with too few copies of $(C,\sigma)$, contradicting \ref{thm:cycleabundance-itm1}.

    \ref{thm:cycleabundance-itm2} $\Rightarrow$ \ref{thm:cycleabundance-itm1}, \ref{thm:cycleabundance-itm3}: Let $C = v_1 \dots v_\ell$. We claim that there exist $i \neq j$ with $\sigma(v_i) = \sigma(v_j)$. Indeed, otherwise we could choose $c$ in such a way that $c(\sigma(v_i)) = i$ for all $i \in [\ell]$, but then the equation $\Eq(C, \sigma, c)$ is $\sum_{i=1}^{\ell-1} x_i - (\ell-1) x_\ell = 0$ which is a convex equation, a contradiction. Note that $v_i$ and $v_j$ cannot be adjacent because $\sigma(v_i) = \sigma(v_j)$, and so $C$ consists of two paths $P$ and $Q$ between $v_i$ and $v_j$, each of length at least $2$.

    Given any injection $c \colon V(F) \to [n]$, we now conclude that $\Eq(C, \sigma, c)$ has genus at least two by partitioning the variables into the two sets of variables corresponding to the edges in $P$ and $Q$ respectively. This proves \ref{thm:cycleabundance-itm3}.

    To obtain \ref{thm:cycleabundance-itm1}, we use \cref{lem:glue_blowup_abundant}. Without loss of generality we may assume that $v_i v_{i+1} \in P$ and $v_{i-1} v_i \in Q$. Observe that $(v_i v_{i+1} v_j v_{i-1}, \sigma)$ is an $F$-coloured $C_4$ that is  $F$-abundant by \cref{lem:peel_abundant} 
    (or indeed by Cauchy-Schwarz if one is so inclined). 
    Furthermore, $(P \setminus v_i, \sigma)$ and $(Q \setminus v_i, \sigma)$ are $F$-abundant. We now replace the edge $v_{i+1} v_j$ by $(P \setminus v_i, \sigma)$ and the edge $v_{i-1} v_j$ by $(Q \setminus v_i, \sigma)$. The resulting graph $(C, \sigma)$ is then $F$-abundant by \cref{lem:glue_blowup_abundant}, as required.

    \ref{thm:cycleabundance-itm3} $\Rightarrow$ \ref{thm:cycleabundance-itm2}: This is true because an equation with genus at least two cannot be convex.
\end{proof}

In the case of triangle-abundance, we show that even if $H$ is not a cycle, we can sometimes derive triangle-abundance (and in fact something much stronger) from properties of $\Eqs(H, \sigma)$. In this setting, an $F$-colouring $\sigma$ of $H$ simply corresponds to a $K_3$-colouring of $H$ with colours $1$, $2$, and $3$. Let a \defn{colour homomorphism} be a homomorphism between two coloured graphs $(H, \sigma)$ and $(H', \sigma')$ that preserves colours. Let $P_\infty^3 = (P_\infty, \phi)$ denotes the infinite path $P_\infty$ with a cyclic $3$-colouring $\phi$ of its vertices.
Note that if $(H, \sigma)$ is colour homomorphic to $P_\infty^3$, then $(H, \sigma)$ is triangle-abundant by \cref{lem:peel_abundant}. We prove the following.

\begin{theorem}\label{lem:symmetric_is_path_hom}
    Let $(H, \sigma)$ be a $K_3$-coloured graph. Then $\Eqs(H, \sigma)$ consists solely of symmetric equations if and only if $(H, \sigma)$ is colour homomorphic to $P_\infty^3$.
\end{theorem}

\begin{proof}
    A \defn{walk} $W$ in a graph is a sequence of vertices $v_0, \dots, v_\ell$ such that $v_{i-1} v_i$ is an edge for every $i \in [\ell]$. We define the \defn{wrap} of an oriented edge $u v$ to be $1$ if $\sigma(u) + 1 \equiv \sigma(v) \pmod{3}$ and $-1$ otherwise. Let $\wrap(W) = \sum_{i=1}^\ell \wrap(v_{i-1} v_i)$. A cycle $C$ is \defn{wrapped}\footnote{Any homotopically inclined reader is free to reintepret this.} if $\wrap(C) \neq 0$.

    \begin{claim}\label{clm:wrapcolhom}
        If $(W, \sigma)$ is a $K_3$-coloured walk, then there is at most one colour homomorphism $\gamma\colon (W, \sigma) \to P_\infty^3$ (up to shifts of $P_\infty^3$). Moreover, if $\gamma$ exists, then $\wrap(W) = \wrap(\gamma(W))$.
    \end{claim}

    \begin{proof}
        Note that for every colour $c$, every vertex in $P_\infty^3$ has at most one neighbour of colour $c$. Thus, once we pick $\gamma(v_0)$ (this choice is unique up to shifts of $P_\infty^3$), this determines $\gamma(v_i)$ for all $i \in [\ell]$ as the unique neighbour of $\gamma(v_{i-1})$ with colour $\sigma(v_i)$. Hence, there can be at most one colour homomorphism $\gamma\colon (W, \sigma) \to P_\infty^3$ (up to shifts of $P_\infty^3$).

        Moreover, if $\gamma$ exists, we have $\sigma(v_i) = \phi(\gamma(v_i))$ for all $i$, so $\wrap(v_{i-1} v_i) = \wrap(\gamma(v_{i-1}) \gamma(v_i))$ and thus $\wrap(W) = \wrap(\gamma(W))$.
    \end{proof}

    \begin{claim}\label{clm:duality}
        Let $(H, \sigma)$ be a $K_3$-coloured graph. Then, $(H, \sigma)$ is colour homomorphic to $P_\infty^3$ if and only if no cycle in $H$ is wrapped.
    \end{claim}

    \begin{proof}
        Suppose that some cycle $C$ in $H$ is wrapped. If $(H, \sigma)$ is colour homomorphic to $P_\infty^3$, this yields a colour homomorphism $\gamma\colon (C, \sigma) \to P_\infty^3$. By \cref{clm:wrapcolhom}, we know that $\wrap(C) = \wrap(\gamma(C))$. However, $\gamma(C)$ is a walk from a vertex back to itself, and such walks have wrap $0$ in $P_\infty^3$, contradicting $\wrap(C) \neq 0$.

        Conversely, suppose that no cycle in $H$ is wrapped. Let $v \in V(H)$ be arbitrary. We may identify the vertices of $P_\infty^3$ with $\bZ$ such that $\phi(n) \equiv \phi(n-1) + 1 \pmod{3}$ for all $n \in \bZ$, and we may assume that $\sigma(v) = \phi(0)$. Define $\gamma$ by $\gamma(u) = \wrap(W)$ where $W$ is an arbitrary walk in $H$ from $v$ to $u$. This is well defined because two distinct walks from $v$ to $u$ with different wrap would imply the existence of a wrapped cycle. It is straightforward to check that $\gamma$ is a colour homomorphism.
    \end{proof}

    \begin{claim}\label{clm:wrapnotsym}
        Let $(C, \sigma)$ be a $K_3$-coloured cycle. Then, $(C, \sigma)$ is wrapped if and only if some equation in $\Eqs(C, \sigma)$ is not symmetric.
    \end{claim}

    \begin{proof}
        Suppose that $(C, \sigma)$ is wrapped. Let $c: [3] \to [n]$ be the identity and consider the equation $\Eq(C, \sigma, c)$. If $\wrap(C) > 0$, then the coefficients $c(2) - c(1)$, $c(3) - c(2)$, and $c(1) - c(3)$ appear more often than their inverses in $\Eq(C, \sigma, c)$, and so there are more positive than negative coefficients. The converse holds if $\wrap(C) < 0$. Therefore, the equation cannot be symmetric since symmetric equations have an equal number of positive and negative coefficients.

        Conversely, suppose that $(C, \sigma)$ is not wrapped. Let $c: [3] \to [n]$ be an arbitrary injection. Since $\wrap(C) = 0$, the coefficients $c(2) - c(1)$, $c(3) - c(2)$, and $c(1) - c(3)$ appear exactly as often as their inverses in $\Eq(C, \sigma, c)$. Hence, $\Eq(C, \sigma, c)$ is symmetric, and because $c$ was arbitrary, every equation in $\Eqs(C, \sigma)$ must be symmetric.
    \end{proof}

    By \cref{clm:wrapnotsym} we know that every equation of $\Eqs(H, \sigma)$ is symmetric if and only if no cycle in $H$ is wrapped. However, \cref{clm:duality} shows that this is equivalent to $(H, \sigma)$ being colour homomorphic to $P_\infty^3$. This concludes the proof.
\end{proof}

\begin{remark}
    Interestingly in the case of $K_3$-coloured graphs, the only graphs for which the Tur\'{a}n \emph{degree} density\footnote{For $(G,\sigma)$ this is defined as the infimum $\delta>0$ such that any  $K_3$-coloured graph $H$ of sufficiently large order with minimum degree (between each pair of parts) at least $\delta |H|$ contains a copy of $(G,\sigma)$.} is zero are those colour homomorphic to $P^3_\infty$.\footnote{To see it is positive for other graphs, let $H$ be the tripartite graph with vertex set $([N],[N],[N])$, and connect two vertices $x,y$ if $y-x \mod N$ is less than $\eps N$. The resulting graph has no short wrapped cycles.}
\end{remark}

This concludes what we can prove, but two central directions remain. 
The first is to find an example of a tripartite triangle-free graph that is not $K_3$-abundant.
We have shown that one cannot do this through the Ruzsa-Szemer\'{e}di construction with convex equations, but the graph $G_3$ would be a good candidate for investigation.
The second question is understanding in what cases the abundance of $\cE(H,\sigma)$ determines the abundance of $(H,\sigma)$ (the opposite implication is given by the Ruzsa-Szemer\'{e}di construction). 

In both cases a new method of constructing non-abundant graphs would be hugely helpful. 
However, it is unclear how difficult this is. 
Another approach is to extend Ruzsa's work, either by understanding $R_E(N)$ for a larger class of equations or by understanding better systems of equations. 
This latter undertaking was begun by Shapira \cite{shapira2006behrend}.
Concretely, in an attempt to extend \cref{lem:symmetric_is_path_hom} beyond symmetric equations,  one could ask
 whether or not there are genus-type properties of \emph{systems} of equations that, if satisfied by all elements of $\cE(G,\sigma)$,
would imply that $(G,\sigma)$ is $F$-abundant. 
See the following section for all open questions.

\section{Open questions}\label{sec:questions}

Thus far we can only show that $H$ is $F$-abundant by showing that some $F$-coloured $(H, \sigma)$ is $F$-abundant. 
While this is clearly sufficient, we conjecture that is also necessary.

\begin{conjecture}\label{conj:reduction_to_coloured}
    A graph $H$ is $F$-abundant if and only if there is an $F$-colouring $\sigma$ of $H$ such that $(H, \sigma)$ is $F$-abundant.
\end{conjecture}

We say a \defn{family $\cF$ of graphs is triangle-abundant} if there exists $C>0$ such that in any graph $G$ that contains at least $\eps n^2$ disjoint triangles, some $F\in \cF$ has density at least $\eps^C$.
The following is much stronger than \cref{conj:reduction_to_coloured}, but we are not confident enough to call it a conjecture. 
Perhaps examples of non-compactness in generalised Tur\'{a}n numbers\footnote{The maximum number of copies of $G$ in an $n$-vertex $F$-free graph. } \cite{alon2016many} could inspire counter-examples.

\begin{question}\label{conj:abundant_families}
    Is it true that a  family $\cF$ of (coloured) graphs is triangle-abundant if only if $\cF$ contains a triangle-abundant (coloured) graph.
\end{question}

Returning to the classification  of abundance,
a major open question is whether or not you can glue graphs on monochromatic vertex sets (if so then all \defn{pinchable} graphs are $K_3$-abundant where pinchable means every subgraph has a monochromatic vertex cut set). 
The most basic open case is for cut sets of size two.

\begin{question}\label{qu:2cutset}
    Suppose $(F, \sigma)$ and $(F', \sigma')$ are triangle-abundant. Let $u, v \in F$ and $u', v' \in F'$ with $\sigma(u) = \sigma(v) = \sigma(u') = \sigma(v')$. If we identify $u$ with $u'$ and $v$ with $v'$ is the resulting graph triangle-abundant?
\end{question}

\begin{remark}
    It is essential that $\sigma(u) = \sigma(v) = \sigma(u') = \sigma(v')$. For example, $(C_4, \sigma)$ is not $K_4$-abundant where $C_4 = x_1 x_2 x_3 x_4$ and $\sigma(x_i) = i$.
\end{remark}

On the connection between equations and graphs,
we re-iterate the (admittedly informal) challenge to extend \cref{lem:symmetric_is_path_hom} to the case where $\cE$ satisfies some genus condition, rather than a symmetry condition. 
The work of Shapira ~\cite{shapira2006behrend} on systems of equations may be useful.
The following remarkable possibility also remains.

\begin{question}
    Is it true that an $F$-coloured graph $(G,\sigma)$ is $F$-abundant if and only if all systems $S\in \cE(G,\sigma)$ are abundant.
\end{question}

If it is the case that the only unavoidable equations are equations of genus at least $2$ (i.e. if there is a positive answer to Ruzsa's question), then it is plausible that the methods contained in this paper yield all abundant graphs.
If on the other hand the only avoidable equations are convex equations, then it seems likely that there are many more (kinds of) abundant graphs, and it may even be that the only non-abundant graphs are increasing-cycle-unavoidable graphs.
If the latter holds then all tripartite triangle-free graphs are triangle-abundant. 
We thus reiterate the question of Gishboliner, Shapira, and Wigderson \cite{GSW23} which asks for a single tripartite triangle-free graph that can be (unconditionally) shown not to be triangle-abundant.

In the additive context, determining $R_E(N)$ for any non-convex equation $E$ of genus one remains wide open. 
Any progress would be exciting.
For example, the following equation has genus one and is not convex. Its solutions are sets of three numbers and two numbers whose averages are equal.

\begin{question}
    Let $E$ be given by $2x_1+2x_2+2x_3 = 3x_4 +3x_5$. 
    Is $E$ avoidable?
\end{question}

Finally, the following additive question is both relevant and interesting in its own right, and may be solvable without classifying any abundant and avoidable equations. 
It has a super-saturation flavour to it.
We recall that an equation $E$ in $\ell$ variables is \defn{abundant} if there exists constants  $\gamma,C>0$ such that any subset $\cA\subset [N]$ of order  $\eps N$ contains at least $\gamma\eps^C N^{\ell-1}$ solutions to $E$, and that it is \defn{avoidable} if $R_E(N) = N^{1-o(1)}$.

\begin{question}
    For an equation, does unavoidability imply abundance?
\end{question}

\section*{Acknowledgements}

The second author would like to thank Yuval Wigderson for a stimulating talk (on \cite{GSW23}) and encouraging conversations at RSA 2023. 
The authors would also like to thank Boris Bukh and Thomas Bloom for informative discussions on the current state of Ruzsa's project. 
{
\fontsize{11pt}{12pt}
\selectfont
	
\hypersetup{linkcolor={red!70!black}}
\setlength{\parskip}{2pt plus 0.3ex minus 0.3ex}

}

\end{document}